\documentclass[reqno]{amsart}
\usepackage{amsmath}%
\usepackage{amsfonts}%
\usepackage{amssymb}%
\usepackage{graphicx}
\newtheorem{theorem}{Theorem}
\theoremstyle{plain}

\newtheorem{corollary}{Corollary}

\newtheorem{definition}{Definition}
\newtheorem{example}{Example}

\newtheorem{lemma}{Lemma}

\newtheorem{proposition}{Proposition}
\newtheorem{remark}{Remark}

\numberwithin{equation}{section}
\begin{document}
\begin{center}
\vspace*{1.3cm}

\textbf{FUNCTIONS WITH UNIFORM SUBLEVEL SETS AND\\SCALARIZATION IN LINEAR SPACES}

\bigskip

by

\bigskip

PETRA WEIDNER\footnote{HAWK Hildesheim/Holzminden/G\"ottingen University of Applied Sciences and Arts, Faculty of Natural Sciences and Technology,
D-37085 G\"ottingen, Germany, {petra.weidner@hawk-hhg.de}.}

\bigskip
\bigskip
Research Report \\ 
Version 3 from December 04, 2017\\
Extension of Version 1 from August 12, 2016

\end{center}

\bigskip
\bigskip

\noindent{\small {\textbf{Abstract:}}
Functions with uniform sublevel sets can represent orders, preference relations or other binary relations and thus turn out to be a tool for scalarization that can be used in multicriteria optimization, decision theory, mathematical finance, production theory and operator theory. Sets which are not necessarily convex can be separated by functions with uniform sublevel sets.  
This report focuses on properties of real-valued and extended-real-valued functions with uniform sublevel sets which are defined on a  linear space without assuming topological properties. The functions may be convex or sublinear. They can coincide with a Minkowski functional or with an order unit norm on a subset of the space.
The considered functionals are applied to the scalarization of vector optimization problems. These vector optimization problems refer to arbitrary domination sets. The consideration of such sets is motivated by their relationship to  domination structures in decision making.
}

\bigskip

\noindent{\small {\textbf{Keywords:}} 
uniform sublevel sets; scalarization; separation theorems; vector optimization; production theory; mathematical finance 
}

\bigskip

\noindent{\small {\textbf{Mathematics Subject Classification (2010): }
46A99, 46N10, 90C29, 90B30, 91B99}}

\section{Introduction}

In this paper, we investigate properties of extended-real-valued functions with uniform sublevel sets. These functions are defined on linear spaces, and the uniform sublevel sets can be described by a linear shift of a set into a specified direction.
These functionals turn out to be of the type $\varphi_{A,k}$, which is defined by
\begin{equation}
\varphi_{A,k} (y):= \inf \{t\in
{\mathbb{R}} \mid y\in A+tk\}, \label{first-funcak0}
\end{equation}
where $A$ is a subset of a linear space $Y$ and $k\in Y\setminus\{0\}$.

This formula was introduced by Tammer (formerly Gerstewitz and Gerth) for convex sets $A$ under more restrictive assumptions in the context of vector optimization \cite{ger85}.
Basic properties of $\varphi _{A,k}$ have been proved in \cite{GerWei90} and \cite{Wei90}, later followed by
\cite{GopRiaTamZal:03}, \cite{TamZal10} and \cite{DT}. Many published results in vector optimization and functional analysis are based on \cite{GerWei90} and \cite{Wei90}, where the functional was used in separation theorems for nonconvex sets and applied to scalarization in vector optimization. 

Because of the strong connection to partial orders, which will be pointed out in Section \ref{sec-Monot}, functions of type $\varphi _{A,k}$ have been used in proofs in different fields of mathematics for the construction of sublinear functionals. In these cases, $A$ is a closed pointed convex cone, usually the ordering cone of the space considered, and $k\in -A$. Among the earliest references listed in \cite{Ham12} are \cite{Bon54} and \cite{Kra64}, where the functional was applied in operator theory. $\varphi _{A,k}$ has also been studied in economic theory and finance, e.g. as so-called shortage function by Luenberger \cite{luen92} and for risk measures by Artzner et al. \cite{art99}. 

In distinction to previous results, we will investigate $\varphi_{A,k}$ without any topological assumptions.

Depending on the choice of $A$ and $k$, $\varphi _{A,k}$ can be real-valued or also attain the value $-\infty$.
We will use the symbolic function value $\nu $ (instead of the value $+\infty$ in convex analysis) when extending a functional to the entire space or at points where a function is not feasible otherwise. Thus our approach differs from the classical one in convex analysis in these cases since the functions we are studying are of interest in minimization problems as well as in maximization problems. Consequently, we consider functions that can attain values in $\overline{\mathbb{R}}_{\nu }:=\overline{\mathbb{R}}\cup\{\nu \}$, where $\overline{\mathbb{R}}:=\mathbb{R}\cup\{-\infty, +\infty\}$. $\varphi _{A,k}$ never attains the value $+\infty$ since we define $\operatorname*{sup}\emptyset = \operatorname*{inf}\emptyset =\nu $. Details of functions with values in $\overline{\mathbb{R}}_{\nu }$ are explained in \cite{Wei15}. 
For the application of this approach to $\varphi _{A,k}$ we have to keep in mind the following terms and definitions:
\begin{enumerate}
\item $\operatorname*{inf}\emptyset =\nu\not\in\overline{\mathbb{R}}$
\item $\operatorname*{dom}\varphi _{A,k}=\{y\in Y\mid \varphi _{A,k}(y)\in\mathbb{R}\cup (-\infty )\}$ is the (effective) domain of $\varphi _{A,k}$
\item $\varphi _{A,k}$ is proper iff $\operatorname*{dom}\varphi _{A,k}\not= \emptyset$ and $\varphi _{A,k}(y)\in\mathbb{R}
\mbox{ for all } y\in \operatorname*{dom}\varphi _{A,k}$
\item $\varphi _{A,k}$ is finite-valued iff $\varphi _{A,k}(y)\in\mathbb{R} \mbox{ for all } y\in Y$ 
\end{enumerate}

We will start our investigations in Section \ref{sec-basicdef} with functions for that the sublevel sets are just linear shifts of a set $A$ into direction $k$ and $-k$. These functions turn out to be of type $\varphi_{A,k}$ with $k\in -0^+A\setminus\{0\}$, where $0^+A$ denotes the recession cone of $A$ defined below.
We investigate basic properties of functionals $\varphi_{A,k}$ defined on arbitrary linear spaces. $\varphi_{A,k}$ is finite-valued if $k\in -\operatorname*{core}0^+A$.

We will always try to find conditions that are sufficient and necessary for certain properties of $\varphi _{A,k}$, e.g. for convexity or sublinearity.
Assumptions are often formulated using the recession cone of $A$. We will show that these assumptions are equivalent to usual assumptions in production theory like the free-disposal assumption.
Proposition \ref{p-varphi_rec} connects $\varphi _{A,k}$ with the sublinear function $\varphi _{0^+A,k}$. 
Proposition \ref{t-sep-all} points out the way in that $\varphi _{A,k}$ separates sets.

Section \ref{sec-Monot} deals with the monotonicity of $\varphi _{A,k}$ in the framework of scalarizing binary relations.
Interdependencies between the functions $\varphi _{A,k}$, $\varphi _{A,\lambda k}$, $\varphi _{A+ck,k}$ and $\varphi _{y^0+A,k}$, which are essential for applications, are studied in Section \ref{sec-var-Ak}. 
Section \ref{sec-cx} focuses on convex functions $\varphi _{A,k}$ including statements for sublinear functionals.
There $A$ is assumed to be a convex cone or a shifted convex cone.
We show the relationship between $\varphi _{A,k}$ and the Minkowski functional of $A+k$ and the coincidence of values of order unit norms with values of $\varphi _{A,k}$.

For applying the functionals with uniform sublevel sets to the scalarization in vector optimization, we start in Section \ref{s-decision} by introducing efficient elements in the framework of decision making. Interdependencies between non-dominated elements w.r.t. relations and efficient elements w.r.t. sets are proved. In Section \ref{s-basics-vo}, we define the vector optimization problem and list some basic results for the efficient point set and the weakly efficient point set. Sufficient conditions for efficiency and weak efficiency are given using minimal solutions of scalar-valued functions. Functionals with uniform sublevel sets are applied for a full characterization of the efficient point set and the weakly efficient point set in Section \ref{s-scal-uni}.Consequences for the scalarization by norms are investigated in Section \ref{s-scal-norm}. The statements extend results from \cite{Wei83}, \cite{Wei85}, \cite{GerWei90} and \cite{Wei90}. 

From now on, $\mathbb{R}$ and $\mathbb{N}$ will denote the sets of real numbers and of non-negative integers, respectively.
We define $\mathbb{N}_{>}:=\mathbb{N}\setminus\{0\}$, $\mathbb{R}_{+}:=\{x\in\mathbb{R}\mid x\geq 0\}$, $\mathbb{R}_{>}:=\{x\in\mathbb{R}\mid x > 0\}$,
$\mathbb{R}_{+}^n:=\{(x_1,\ldots ,x_n)\in\mathbb{R}^n\mid x_i\geq 0 \mbox{ for all } i\in\{1,\ldots, n\}\}$ for each $n\in\mathbb{N}_>$.
Linear spaces will always be assumed to be real vector spaces. 
A set $C$ in a linear space $Y$ is a cone iff $\lambda c\in C \mbox{ for all }\lambda\in\mathbb{R}_{+}, c\in C$. The cone $C$ is called non-trivial iff $C\not=\emptyset$, $C\not=\{0\}$ and $C\not= Y$ hold. For a subset $A$ of some linear space $Y$, 
$\operatorname*{core}A$ will denote the algebraic interior of $A$ and $0^+A:=\{u\in Y  \mid  \forall a\in A \; \forall t\in \mathbb{R}_{+}\colon a+tu\in A\ \}$ the recession cone of $A$.
Given two sets $A$, $B$ and some vector $k$ in $Y$, we will use the notation $A\; B:=A \cdot B:=\{a \cdot b\mid a \in A , \; b\in B\}$ and $A\; k:=A \cdot k:=A \cdot \{ k\}$.
For a functional $\varphi$ defined on some space $Y$ and attaining values in $\overline{\mathbb{R}}_{\nu }$, we will denote the epigraph of $\varphi$ by   
$\operatorname*{epi}\varphi$, the effective domain of $\varphi$ by $\operatorname*{dom}\varphi$, and, with respect to some binary relation $\mathcal{R}$ given on $\overline{\mathbb{R}}_{\nu }$, we consider the sets $\operatorname*{lev}_{\varphi,\mathcal{R}}(t):= \{y\in Y \mid \varphi (y) \mathcal{R} t\}$ with $t\in\mathbb{R}$.

Throughout the paper, $Y$ will be a linear space, and we assume $A$ to be a nonempty subset of $Y$ and $k\in Y\setminus\{0\}$. $\mathcal{P}(Y)$ denotes the power set of $Y$.

Beside the properties of functions defined in \cite{Wei15}, we will need the following ones:
\begin{definition}\label{d-mon}
Assume $B\subseteq Y$ and $\varphi: Y \to \overline{\mathbb{R}}_{\nu } $. \\
$\varphi$ is said to be
\begin{itemize}
\item[(a)]
$B$-monotone on $F\subseteq \operatorname*{dom}\varphi$
iff $y^1,y^2 \in F$ and $y^{2}-y^{1}\in B$ imply $\varphi
(y^{1})\le \varphi (y^{2})$,
\item[(b)] strictly $B$-monotone on $F\subseteq \operatorname*{dom}\varphi$ 
iff $y^1,y^2 \in F$ and $y^{2}-y^{1}\in B\setminus
\{0\}$ imply $\varphi (y^{1})<\varphi (y^{2})$,
\item[(c)] $B$-monotone or strictly $B$-monotone iff it is $B$-monotone or strictly $B$-monotone, respectively, on $\operatorname*{dom}\varphi$,
\item[(d)] quasiconvex iff $\operatorname*{dom}\varphi$ is convex and
\[\varphi (\lambda y^1 + (1-\lambda ) y^2) \le \operatorname*{max}(\varphi (y^1),\varphi (y^2))\]
for all $y^1, y^2 \in \operatorname*{dom}\varphi$ and $\lambda \in (0,1)$.
\end{itemize}
\end{definition}

The following lemmata, where the first one is due to \cite{Zal:86a}, will be used in proofs.

\begin{lemma}\label{cor_finite_Zal}
Let $C\subseteq Y$ be a convex cone.
Then $Y=C+\mathbb{R}k$ holds if and only if
\begin{itemize}
\item[(a)] $C$ is a linear subspace of $\,Y$ of codimension $1$ and $k\not\in C$, or
\item[(b)] $\{k,-k\}\cap \operatorname*{core}C\not=\emptyset$. 
\end{itemize}
\end{lemma}

\begin{lemma}\label{cone_gen_noncx}
Let $C\subseteq Y$ be a cone and $k\in -\operatorname*{core}C$. Then $Y=C+\mathbb{R}_{>}k$. 
\end{lemma}
\begin{proof}
Consider some arbitrary $y\in Y$. $\Rightarrow \exists t\in\mathbb{R}_{>}:\; k+t(-y)\in -C$ since $k\in -\operatorname*{core}C$. $\Rightarrow y\in C+\mathbb{R}_{>}k$ since $C$ is a cone.
\end{proof}
\smallskip

\section{Definition and Basic Properties of Functions with Uniform Sublevel Sets}\label{sec-basicdef}

Scalarization is closely linked to separation. A functional $\varphi$ separates two sets $V$ and $W$ in the space $Y$ if there exists some value $t\in \mathbb{R}$ such that one of the sets is contained in $M:=\{y\in Y \mid \varphi (y) \le t \}$, the other one is contained in $\{y\in Y \mid \varphi (y) \ge t \}$ and $V\cup W\not\subseteq \{y\in Y \mid \varphi (y) = t \}$. 
Disjoint convex sets in a finite-dimensional vector space can be separated by some linear functional $\varphi$. In this case, $M=tk+A$ for some halfspace A and some $k\in Y$. Being interested in nonconvex sets, we use this idea and investigate functionals $\varphi $ that fulfill the condition
\begin{equation} \label{K4}
\forall t\in\mathbb{R}:\quad\varphi(y)\leq t \iff y\in A+tk.
\end{equation}
Here, A is assumed to be an arbitrary subset of Y.

The construction of functions with uniform sublevel sets is based on the following proposition.

\begin{proposition} \label{prop-vor-theo}
Consider a function $\varphi:Y\rightarrow \overline{{\mathbb{R}}}_{\nu }$ with $\operatorname*{dom}\varphi =A+\mathbb{R}k$.\\
If (\ref{K4}) is satisfied, then
\begin{equation}\label{vor-funcak0}
\varphi (y)= \inf \{t\in
{\mathbb{R}} \mid y\in A+tk \} \mbox{ for all }  y\in Y.
\end{equation}
Moreover, (\ref{K4}) is equivalent to
\begin{equation}\label{epi_Ak}
\operatorname*{epi}\varphi =\{(y,t)\in Y\times \mathbb{R} \mid y\in A+tk\}.
\end{equation} 
\end{proposition}
\begin{proof}
(\ref{K4}) and $\operatorname*{dom}\varphi = A+\mathbb{R}k$ imply $\varphi (y)\not= +\infty\mbox{ for all } y\in Y$.
If $\varphi (y)=-\infty$, then $y\in tk+A\mbox{ for all } t \in {\mathbb{R}}$, thus $\inf \{t\in
{\mathbb{R}} \mid y\in tk + A\}=-\infty$. If $\varphi (y)=t\in {\mathbb{R}}$, then $y\in tk+A$. If (\ref{vor-funcak0}) would not be satisfied, then there would exist some 
$\lambda \in {\mathbb{R}}$ with $\lambda < t$ and 
$y\in \lambda k+A$. This would imply $\varphi (y)\le\lambda < t$, a contradiction. Hence (\ref{vor-funcak0}) holds.
The second statement is obvious.
\end{proof}

Hence each functional with uniform sublevel sets is of the following type.

\begin{definition}\label{d-funcak0}
The function $\varphi _{A,k}:Y\rightarrow \overline{{\mathbb{R}}}_{\nu }$ is defined by
\begin{equation}
\varphi_{A,k} (y):= \inf \{t\in
{\mathbb{R}} \mid y\in A+tk\}. \label{funcak0}
\end{equation}
\end{definition}

One gets an immediate geometric interpretation of $\varphi_{A,k}$ since $A+tk$ is just the set $A$ shifted by $tk$.

We will now investigate basic properties of the functional $\varphi_{A,k}$. Functions with uniform sublevel sets will be characterized as functions of type $\varphi_{A,k}$ that have two additional properties.

\begin{definition}
$A$ is said to be $k$-directionally closed if
\begin{equation}\label{k-dir-clsd}
\forall y\in Y:\quad((\exists (t_n)_{n\in\mathbb{N}}:\; t_n\searrow 0 \mbox{ and } y-t_nk\in A) \Rightarrow y\in A).
\end{equation} 
\end{definition}

\begin{theorem}\label{varphi-theo-allg}
\begin{eqnarray}
\operatorname*{dom}\varphi _{A,k} & = & A+\mathbb{R}k,\label{dom_Ak}\\
\varphi _{A,k} (y+t k) & = & \varphi _{A,k} (y)+t \quad \mbox{ for all } y\in Y,\, t \in\mathbb{R},
\label{f-r255nn}\\
\varphi _{A,k}(y) & \leq & t\quad\mbox{ for all } t \in\mathbb{R},\,y\in A+tk,\label{A_in_leq}\\
\varphi _{A,k}(y) & < & t\quad\mbox{ for all } t \in\mathbb{R},\,y\in \operatorname*{core}A+tk.\label{core_in_less}
\end{eqnarray}
\begin{itemize}
\item[]
\item[(a)]
The condition
\begin{equation}\label{less_in_A}
\operatorname*{lev}\nolimits_{\varphi_{A,k},<}(t)\subseteq A+tk \mbox{ for all } t \in\mathbb{R}
\end{equation}
holds if and only if $k\in -0^+A$.
\item[(b)] If $A$ is $k$-directionally closed, then
\begin{equation}\label{eq_in_A}
\operatorname*{lev}\nolimits_{\varphi_{A,k},=}(t)\subseteq A+tk \mbox{ for all } t \in\mathbb{R}.
\end{equation}
\item[(c)] 
The condition
\begin{equation}\label{K4-ak}
\operatorname*{lev}\nolimits_{\varphi_{A,k},\leq}(t)=A+tk \quad \mbox{ for all } t \in\mathbb{R}
\end{equation}
is fulfilled if and only if the following conditions hold:
\begin{equation}\label{d-sep-func}
k\in -0^+A 
\end{equation}
and $$A \mbox{ is }k\mbox{-directionally closed.}$$
\item[(d)] 
The following conditions are equivalent to each other:
\begin{equation} A-\mathbb{R}_{>}\cdot k\subseteq \operatorname*{core} \;A, \label{in_core} \end{equation} 
\begin{equation}\label{core_in_equal} 
\operatorname*{lev}\nolimits_{\varphi_{A,k},<}(t)= \operatorname*{core}A +tk \quad \mbox{ for all } t \in \mathbb{R}.
\end{equation} 
Condition (\ref{in_core}) implies $\operatorname*{dom}\varphi _{A,k}=\operatorname*{core}\operatorname*{dom}\varphi _{A,k}$.
\item[(e)] If $A\subseteq A_0\subseteq Y$, then
$\operatorname*{dom}\varphi _{A,k}\subseteq \operatorname*{dom}\varphi _{A_0,k}$, and 
\begin{equation*}
\varphi _{A_0,k}(y)\leq \varphi _{A,k}(y)\quad\mbox{ for all } y\in\operatorname*{dom}\varphi _{A,k}.
\end{equation*}
\item[(f)] If $k\in 0^+A$, then $\varphi _{A,k}(y)=-\infty\mbox{ for all } y\in\operatorname*{dom}\varphi _{A,k}$.
\end{itemize}
\end{theorem} 
\begin{proof}
(\ref{dom_Ak}), (\ref{f-r255nn}), (\ref{A_in_leq}), (\ref{core_in_less}) and (e) are obvious.
\begin{itemize}
\item[(a)] First, assume that (\ref{less_in_A}) is fulfilled.
For each $y\in A-tk$ with $t\in\mathbb{R}_>$, we get $\varphi _{A,k}(y)\le -t < 0$ and thus $y \in A$. Hence $k\in -0^+A$.\\
Conversely, assume $k\in -0^+A$.
Consider $y\in Y$ and $t\in\mathbb{R}$ with $\varphi  _{A,k}(y) < t$. Then there exists some $\lambda\in \mathbb{R}_>$ with $y\in A+ (t-\lambda )k=A-\lambda k+tk$. Hence $y\in A+tk$ because of $k\in -0^+A$.
\item[(b)]
Assume that $A$ is $k$-directionally closed, but (\ref{eq_in_A}) is not fulfilled. Then there exist some $y\in Y$ and $t\in\mathbb{R}$ with $\varphi  _{A,k}(y) = t$ and $y\notin A+tk$.
The definition of $\varphi_{A,k}$ implies the existence of a sequence $(t_n)_{n\in\mathbb{N}}$ with $t_n\in\mathbb{R}$, $t_n\searrow 0$ and $y \in A+(t+t_n)k$, i.e., $y-tk-t_nk \in A .$ Hence $y-tk\in A$ since $A$ is $k$-directionally closed, a contradiction to $y\notin A+tk .$
\item[(c)] 
Assume first that (\ref{d-sep-func}) and (\ref{k-dir-clsd}) hold.
We have to prove that $\operatorname*{lev}_{\varphi ,\le}(t) \subseteq t k+A\mbox{ for all } t \in {\mathbb{R}}$.
Consider $y\in Y$ and $t\in\mathbb{R}$ with $\varphi  _{A,k}(y) < t$. Then there exists some $\lambda\in \mathbb{R}_>$ with $y\in A+ (t-\lambda )k=A-\lambda k+tk$. Hence $y\in A+tk$ because of $k\in -0^+A$.
Assume that (\ref{K4-ak}) is not fulfilled. Then there exist some $y\in Y$ and $t\in\mathbb{R}$ with $\varphi  _{A,k}(y) = t$ and $y\notin A+tk$.
The definition of $\varphi_{A,k}$ implies the existence of a sequence $(t_n)_{n\in\mathbb{N}}$ with $t_n\in\mathbb{R}$, $t_n\searrow 0$ and $y \in A+(t+t_n)k$, i.e., $y-tk-t_nk \in A .$ Hence $y-tk\in A$ since $A$ is $k$-directionally closed, a contradiction to $y\notin A+tk .$
Thus (\ref{K4-ak}) is fulfilled.\\
Assume now (\ref{K4-ak}). Then $y\in A-tk$, $t\in\mathbb{R}_{+}$, results in $\varphi(y)\le -t \le 0$ and thus in $y \in A$. Hence (\ref{d-sep-func}) is satisfied.\\
If $y\in Y$ and $y-t_nk\in A$ for some sequence $(t_n)_{n\in\mathbb{N}}$ with $t_n\searrow 0$, then $\varphi (y)\leq t_n\mbox{ for all } n\in\mathbb{N}$ and thus $\varphi (y)\leq 0$, which implies $y\in A$. Thus $A$ is $k$-directionally closed.
\item[(d)] 
Assume that (\ref{in_core}) holds. Let $t \in
{\mathbb{R}}$ and $y\in Y$ be such that $\varphi_{A,k} (y)<t $.
Then there exists some $\lambda \in {\mathbb{R}}$, $\lambda < t $, with $y\in
\lambda k+A$. It follows that $y\in \lambda k+A=t k+(A-(t
- \lambda )k)\subseteq t k+\operatorname*{core}A$. This results, together with (\ref{core_in_less}), in 
(\ref{core_in_equal}).\\
Let us now assume that (\ref{core_in_equal}) is satisfied. Consider some $y\in A-\mathbb{R}_{>}\cdot k$. $\Rightarrow \varphi _{A,k}(y)< 0$. This implies $y\in \operatorname*{core}A$ by (\ref{core_in_equal}). Thus (\ref{in_core}) is fulfilled. \\
If (\ref{in_core}) holds, then $\mathbb{R} k+A\subseteq \mathbb{R} k+\operatorname*{core}A\subseteq \operatorname*{core}(\mathbb{R} k+A)$, i.e.
$\operatorname*{dom}\varphi _{A,k} \subseteq\operatorname*{core}\operatorname*{dom}\varphi _{A,k}$.
\item[(f)] The assertion follows from: $tk+A=(t-\lambda )k+A+\lambda k \subseteq (t-\lambda )k+A\quad\mbox{ for all } t\in\mathbb{R}$.
\end{itemize}
\end{proof}

\begin{remark}
Property (\ref{f-r255nn}) is called translation invariance and plays an important role in several proofs as well as for applications in risk theory. It was shown for $\varphi _{A,k}$ in \cite{GopTamZal:00}. Hamel \cite[Proposition 2]{Ham12}  pointed out that each translation-invariant functional $\varphi :Y\to\overline{\mathbb{R}}$ fulfills property (\ref{K4}) with $A=\{y\in Y\mid \varphi(y)\leq 0\}$. He investigated the relationship between translation invariance and the conditions (\ref{vor-funcak0}), (\ref{d-sep-func}) and (\ref{k-dir-clsd}), where his definition of $k$-directional closedness is different from that used in this paper.
\end{remark}

Let us add some statements related to the conditions mentioned in Theorem \ref{varphi-theo-allg}.

\begin{lemma}\label{l-dir-clsd-suff}
Assume that $A$ is algebraically closed and $k\in -0^+A\setminus\{0\}$.
Then $A$ is $k$-directionally closed.
\end{lemma}
\begin{proof}
Suppose that (\ref{k-dir-clsd}) is not fulfilled. Then there exists some $y\in Y$ and some sequence $(t_n)_{n\in\mathbb{N}}$ with $t_n\searrow 0$ and $y-t_nk\in A$, but $y\notin A$.
Since $A$ is algebraically closed, we get: $\exists\lambda\in (0,1):\,y-t_1k+\lambda (y-(y-t_1k))\notin A$, i.e., $y-t_1k+\lambda t_1k=y-(1-\lambda )t_1k\notin A$. There exists some $j\in \mathbb{N}$ with $t_j<(1-\lambda )t_1$. Since $y-t_jk\in A$, we get $y-(1-\lambda )t_1k= y-t_jk+(t_j-(1-\lambda )t_1)k\in A$ by $k\in -0^+A$, a contradiction.
\end{proof}

The assumption $k\in -0^+A$ can be formulated in different ways.

\begin{proposition}\label{H1-alternat}
Suppose that $A$ is a proper subset of $\,Y$.\\
The following conditions are equivalent to each other  for $A$ and $k\in Y\setminus\{0\}$.
\begin{itemize}
\item[(a)] $k\in -0^+A$.
\item[(b)] $A=H-C$ for some proper subset $H$ of $\,Y$ and some convex cone $C\subset Y$ with $k\in C$.
\item[(c)] $A=A-C$ for some non-trivial convex cone $C\subset Y$ with $k\in C$.
\item[(d)] $A=A-C$ for some non-trivial cone $C\subset Y$ with $k\in C$.
\end{itemize}
\end{proposition}
\begin{proof}
(a) implies (b) with $H=A$ and $C=-0^+A$. (b) implies (c) since $A-C= H-C-C=H-C=A$.
(c) yields (d). (d) implies (a) because of $C\subseteq -0^+A$.
\end{proof}

\begin{remark}
One of the basic assumptions in production theory is the free-disposal assumption $A=A-C$, where $C$ is the ordering cone.
\end{remark}

\begin{proposition}\label{p-k-recc}
Assume $k\in -0^+A$.
\begin{itemize}
\item[(a)] $\operatorname*{dom}\varphi _{A,k}=A+\mathbb{R}_>k$.
\item[(b)] $\varphi  _{A,k} (y)=-\infty \iff y+\mathbb{R}k\subseteq A$.
\item[(c)] $\varphi  _{A,k} $ is finite-valued on $\operatorname*{dom}\varphi  _{A,k} \setminus A$. 
\end{itemize}
\end{proposition} 
{\it Proof}
\begin{itemize}
\item[(a)] We get for each $t\in\mathbb{R}$ with $t\leq 0$: \quad$A+tk=A+k+(t-1)k\subseteq A+k.$
\item[(b)] Consider $y\in \operatorname*{dom}\varphi  _{A,k}$ with $\varphi  _{A,k} (y)=-\infty$ and $t\in\mathbb{R}.$
$\Rightarrow\exists\lambda <-t:\;y\in\lambda k+A.$ $\Rightarrow \lambda +t<0$ and $y-\lambda k\in A.$
$\Rightarrow y+tk=(y-\lambda k)+(\lambda + t)k\in A+ 0^+A\subseteq A$.\\
The reverse direction of the equivalence is obvious.
\item[(c)] (b) implies for $y\in Y$ with $\varphi  _{A,k} (y)=-\infty$: $\,y=y+0\cdot k\in A$.
\qed
\end{itemize}

\begin{theorem}\label{theo-k-core}
Assume that $A$ is a proper subset of $\,Y$ and  $k\in -\operatorname*{core}0^+A.$\\
Then $\varphi  _{A,k}$ is finite-valued and 
(\ref{core_in_equal}) holds.
\end{theorem} 
\begin{proof}
By Lemma \ref{cor_finite_Zal}, $Y=0^+A+\mathbb{R}k$, thus $Y=A+0^+A+\mathbb{R}k \subseteq A+\mathbb{R}k=\operatorname*{dom}\varphi _{A, k}$.\
Suppose now that $\varphi  _{A,k}$ is not finite-valued. Then there exists some $y\in Y$ with $y+\mathbb{R}k \subseteq A$, which implies $Y=0^+A+\mathbb{R}k=0^+A+\mathbb{R}k+y\subseteq 0^+A+A\subseteq A$, a contradiction.\\
Since $A+\operatorname*{core}0^+A\subseteq \operatorname*{core}(A+0^+A) \subseteq\operatorname*{core}A$, (\ref{core_in_equal}) follows from Theorem \ref{varphi-theo-allg}.
\end{proof}

\begin{theorem} \label{t-K4-cons}
Assume $A$ is a proper subset of $\,Y$. 
\begin{itemize}
\item[(I)] 
\begin{itemize}
\item[(a)] If $A$ is convex, then $\varphi  _{A,k} $ is convex.
\item[(b)] If $A$ is a cone, then $\varphi  _{A,k} $ is positively homogeneous.
\item[(c)] If $A+A\subseteq A$, then $\varphi  _{A,k} $ is subadditive.
\item[(d)] If $A$ is a convex cone, then $\varphi  _{A,k} $ is sublinear.
\end{itemize}
\item[(II)] 
Suppose that $A$ is $k$-directionally closed and that $k\in -0^+A\setminus\{0\}$ holds.\\
Then each of the sufficient conditions given in (a)-(d) is also necessary.
Moreover, $\varphi  _{A,k} $ is quasiconvex if and only if $\varphi  _{A,k} $ is convex.
\end{itemize}
\end{theorem}
{\it Proof}
\begin{itemize}
\item[(I)] 
\begin{itemize}
\item[(a)] Take $(y^{1},t_1)$, $(y^{2},t_2)\in \operatorname*{epi}\varphi  _{A,k} $, $\lambda\in[0,1]$. 
For each $i\in\{ 1,2\}$, there exists a sequence
$(t_{in})_{n\in\mathbb{N}}$ of real numbers which converges to $\varphi_{A,k}(y^i)$ with $t_{in}\geq\varphi _{A,k}(y^i)$ and
$y^i\in A+t_{in}k$ for all $n\in\mathbb{N}$.
The convexity of $A$ yields $\lambda y^{1}+(1-\lambda )y^{2}\in A+(\lambda t_{1n}+(1-\lambda) t_{2n})k$ for all $n\in\mathbb{N}$.
Thus $\varphi  _{A,k}(\lambda y^{1}+(1-\lambda )y^{2})\leq \lambda \varphi _{A,k}(y^1)+(1-\lambda) \varphi _{A,k}(y^2)\leq \lambda t_1+(1-\lambda) t_2$.
Hence $\operatorname*{epi}\varphi  _{A,k} $ is convex, i.e., $\varphi  _{A,k} $ is convex.
\item[(b)] 
Take $(y,t)\in \operatorname*{epi}\varphi  _{A,k} $ 
and $\lambda\in\mathbb{R}_{+}$. 
There exists a sequence
$(t_{n})_{n\in\mathbb{N}}$ of real numbers which converges to $\varphi_{A,k}(y)$ with $t_{n}\geq\varphi _{A,k}(y)$ and
$y\in A+t_{n}k$ for all $n\in\mathbb{N}$.
Then $\lambda y\in \lambda A+\lambda t_nk \subseteq A+\lambda t_nk$ for all $n\in\mathbb{N}$. Thus $\varphi  _{A,k}(\lambda y)\leq \lambda\varphi _{A,k}(y)\leq \lambda t$.
Hence $\operatorname*{epi}\varphi  _{A,k} $ is a cone, i.e., $\varphi  _{A,k} $ is positively homogeneous.
\item[(c)] 
Take $(y^{1},t_1),(y^{2},t_2)\in \operatorname*{epi}\varphi  _{A,k} $. 
For each $i\in\{ 1,2\}$, there exists a sequence
$(t_{in})_{n\in\mathbb{N}}$ of real numbers which converges to $\varphi_{A,k}(y^i)$ with $t_{in}\geq\varphi _{A,k}(y^i)$ and
$y^i\in A+t_{in}k$ for all $n\in\mathbb{N}$.
Since $A+A\subseteq A$, $y^{1}+y^{2}\in A+(t_{1n}+t_{2n})k$ for all $n\in\mathbb{N}$.
Thus $\varphi  _{A,k}(y^{1}+y^{2})\leq t_{1n}+t_{2n}$ for all $n\in\mathbb{N}$.
This implies $\varphi  _{A,k}(y^{1}+y^{2})\leq\varphi _{A,k}(y^1)+\varphi _{A,k}(y^2)\leq t_1+t_2$.
Consequently, $\operatorname*{epi}\varphi  _{A,k} +\operatorname*{epi}\varphi  _{A,k} \subseteq\operatorname*{epi}\varphi  _{A,k} $, i.e., $\varphi  _{A,k}$ is subadditive.
\item[(d)] follows from (a) and (b) since a functional is sublinear if and only if it is convex and positively homogeneous.
\end{itemize}
\item[(II)] 
The assumptions imply (\ref{K4-ak}) and
$\operatorname*{epi}\varphi _{A,k} =\{(y,t)\in Y\times \mathbb{R}  : y\in A+tk\}$ by Proposition \ref{prop-vor-theo}.  $\varphi _{A,k}$ is quasiconvex if and only if all sets 
$\operatorname*{lev}\nolimits_{\varphi_{A,k},\leq}(t)$ with $t \in\mathbb{R}$
are convex (see \cite{Wei15}).
This results in the statements. 
\qed
\end{itemize}

The values of $\varphi _{A,k}$ are connected with the values of the sublinear functional $\varphi _{0^+A,k}$.

\begin{proposition}\label{p-varphi_rec}
\begin{itemize}
\item[]
\item[(a)] For $y^0\in A+\mathbb{R}k$ and $y^1\in 0^+A+\mathbb{R}k$, we get
$y^0+y^1\in A+\mathbb{R}k$ and
\begin{equation*}
\varphi_{A,k}(y^0+y^1)\leq \varphi_{A,k}(y^0)+\varphi_{0^+A,k}(y^1).
\end{equation*}
\item[(b)] If $A$ is a proper subset of $Y$ and $k\in -\operatorname*{core}0^+A$, then $\varphi _{A,k}$ and $\varphi _{0^+A, k}$ are finite-valued and
\begin{equation}
\varphi_{A,k}(y^0)-\varphi_{A,k}(y^1)\leq \varphi_{0^+A,k}(y^0-y^1) \mbox{ for all } y^0,y^1\in Y.\label{vor_sublin_dom}
\end{equation}
\item[(c)] If $A$ is $k$-directionally closed, then $0^+A$ is $k$-directionally closed.
\end{itemize}
\end{proposition}
{\it Proof}
\begin{itemize}
\item[(a)] 
Consider $y^0\in A+\mathbb{R}k, y^1\in 0^+A+\mathbb{R}k$.
Then $y^0+y^1\in (A+0^+A)+\mathbb{R}k$. Thus $y^0+y^1\in A+\mathbb{R}k$.
There exists a sequence $(t_n)_{n\in\mathbb{N}}$ of real numbers with $y^0\in A+t_nk$ which converges to 
$\varphi_{A,k}(y^0)$. Furthermore, there exists a sequence $(s_n)_{n\in\mathbb{N}}$ of real numbers with $y^1\in 0^+A+s_nk$ which converges to 
$\varphi_{0^+A,k}(y^1)$. For all $n\in\mathbb{N}$, $y^0+y^1\in (A+t_nk)+(0^+A+s_nk)\subseteq A+(t_n+s_n)k$\\
and thus $\varphi_{A,k}(y^0+y^1)\leq t_n+s_n$ by (\ref{A_in_leq}). Hence\\
$\varphi_{A,k}(y^0+y^1)\leq \operatorname*{lim}_{n\to +\infty}(t_n+s_n)=\varphi_{A,k}(y^0)+\varphi_{0^+A,k}(y^1)$.
\item[(b)] The assumptions imply that $0^+A$ is a proper subset of $Y$. The considered functionals are finite-valued by Theorem \ref{theo-k-core}. (\ref{vor_sublin_dom}) results from (a). 
\item[(c)] Assume that $A$ is $k$-directionally closed and that we have some $y\in Y$ for that there exists some sequence $(t_n)_{n\in\mathbb{N}}$ of real numbers with
$t_n\searrow 0$ and $y-t_nk\in 0^+A$ for all $n\in\mathbb{N}$.
Consider arbitrary elements $a\in A$ and $\lambda\in\mathbb{R}_>$.
Since $0^+A$ is a cone, $\lambda y-t_n\lambda k\in 0^+A\mbox{ for all } n\in\mathbb{N}$. Hence $a+\lambda y-t_n\lambda k\in A\mbox{ for all } n\in\mathbb{N}$.
Since $A$ is $k$-directionally closed and $\lambda t_n\searrow 0$, we get $a+\lambda y\in A$. Thus $y\in 0^+A$.
Hence $0^+A$ is $k$-directionally closed.
\qed
\end{itemize}

\smallskip
The functional $\varphi_{A,k}$ has been constructed in such a way that it can be used for the separation of not necessarily convex sets.

\begin{proposition}\label{t-sep-all}
Assume $D$ to be a nonempty subset of $Y$.
\begin{itemize}
\item[(1)] $(\forall d\in D\colon\varphi_{A,k}(d)\not< 0)\implies \operatorname*{core} A\cap D=\emptyset$.
\item[(2)] If $A-\mathbb{R}_{>}\cdot k\subseteq \operatorname*{core}A$, then :\\
$\operatorname*{core} A\cap D=\emptyset\iff (\forall d\in D\colon\varphi_{A,k}(d)\not< 0) $.
\item[(3)] If $k\in -0^+A\setminus\{0\}$ and $A$ is $k$-directionally closed, then:\\
$A\cap D=\emptyset\iff  (\forall d\in D\colon\varphi_{A,k}(d)\not\leq 0)$.
\end{itemize}
\end{proposition}

\begin{proof}
(1) follows from $\varphi_{A,k}(a) < 0 \mbox{ for all } a\in \operatorname*{core} A$, (2) from (\ref{core_in_equal}),
(3) from (\ref{K4-ak}).
\end{proof}

Since we use $\nu$ as function value outside the effective domain, $\not\leq$ and $\not<$ can only be replaced by $>$ and $\geq$, respectively, if $Y=A+{\mathbb{R}}k$.
\smallskip

\section{Representation of Binary Relations by Functions and Monotonicity}\label{sec-Monot}

Binary relations, especially partial orders, can structure a space or express preferences in decision making and  optimization. Thus the presentation of such relations by real-valued functions serves as a useful tool in proofs, e.g. in operator theory \cite{Kra64}, but also as a basis for scalarization methods in vector optimization \cite{Wei90} and for the development of risk measures in mathematical finance \cite{art99}.

If $C$ is an algebraically closed ordering cone in $Y$, then the corresponding order $\leq _{C}$ can be presented by $\varphi_{-C,k}$ with an arbitrary $k\in C\setminus\{0\}$ since, for all $y^1,y^2\in Y$,
\begin{displaymath}
y^1\leq _{C} y^2\iff \varphi_{-C,k}(y^1-y^2)\leq 0.
\end{displaymath}
Since $\varphi_{-C,k}$ is $C$-monotone, we get for all $y^1,y^2\in \operatorname*{dom}\varphi_{-C,k}$:
\begin{equation}\label{mon-not-reverse}
y^1\leq _{C} y^2\implies \varphi_{-C,k}(y^1)\leq \varphi_{-C,k}(y^2).
\end{equation}

More generally, if a binary relation can be described by some proper algebraically closed subset $A$ of $Y$ with $0^+A\not= \{ 0\}$ as
$\mathcal{R}_A=\{(y^1,y^2)\in Y\times Y\mid y^2-y^1\in A\}$,  
then we have for each $k\in 0^+A\setminus \{ 0\}$ and all $y^1, y^2\in Y$:
\begin{displaymath}
y^1\mathcal{R}_A y^2\iff \varphi_{-A,k}(y^1-y^2)\leq 0.
\end{displaymath}
If the function $\varphi_{-A,k}$ is $A$-monotone, this implies for all $y^1,y^2\in \operatorname*{dom}\varphi_{-A,k}$:
\begin{displaymath}
y^1\mathcal{R}_A y^2 \implies \varphi_{-A,k}(y^1)\leq \varphi_{-A,k}(y^2).
\end{displaymath}

The reverse implication is not true since it is already not true for (\ref{mon-not-reverse}).
\begin{example}
Consider $Y=\mathbb{R}^2$, $C=\mathbb{R}^2_{+}$ and $k=(1,1)^T$. Then $\operatorname*{dom}\varphi_{-C,k}=Y$.
For $y^1=(-1,-1)^T$ and $y^2=(-2,0)^T$, we get
$\varphi_{-C,k}(y^1)=-1\leq 0=\varphi_{-C,k}(y^2)$, but $y^1\leq _{C} y^2$ does not hold since $y^2-y^1=(-1,1)^T\not\in C$.
\end{example}

But we get the following local presentation of $\mathcal{R}_A$. One has for all $y^1, y^2\in Y$:
\begin{displaymath}
y^1\mathcal{R}_A y^2\iff \varphi_{y^2-A,k}(y^1)\leq 0.
\end{displaymath}

\begin{theorem}\label{t251M}
Assume $B\subseteq Y$.
\begin{itemize}
\item[(1)] $A-B\subseteq A\implies\varphi _{A,k}$ is $B$-monotone.
\item[(2)] If $k\in -0^+A$, $A$ is $k$-directionally closed and $A-B\subseteq A+\mathbb{R}k$, then:\\
$\varphi _{A,k} $ is $B${-monotone} $\iff \,A-B\subseteq A$.
\item[(3)] If $A$ is $k$-directionally closed and $\varphi _{A,k}$ is finite-valued on $F\subseteq Y$, then:\\
$A - (B\setminus \{0\}) \subseteq \operatorname{core} A \implies \varphi _{A,k} $ is strictly $B$-monotone on $F$.
\item[(4)] If $A-\mathbb{R}_{>} k\subseteq \operatorname*{core} \;A$ and $A-B\subseteq A+\mathbb{R}k$, then:\\
$\varphi _{A,k} $ is strictly $B${-monotone} $\implies \,A - (B\setminus \{0\}) \subseteq \operatorname{core} A$.
\item[(5)] Suppose that $\varphi _{A,k}$ is finite-valued on $F\subseteq Y$. Then:\\
$A-B\subseteq A\implies\varphi _{A,k} $ is strictly ($\operatorname*{core}B$)-monotone on $F$.
\end{itemize}
\end{theorem}
{\it Proof}
\begin{itemize}
\item[(1)] Suppose $A - B\subseteq A$. Take $y^1,y^2 \in \operatorname*{dom}\varphi_{A,k}$ with $y^{2}-y^{1}\in B$. There exists a sequence
$(t_n)_{n\in\mathbb{N}}$ that converges to $\varphi_{A,k}(y^2)$ such that $y^2\in t_nk+A \mbox{ for all } n\in\mathbb{N}$.
$\Rightarrow y^{1}\in y^{2}-B\subseteq t_nk + (A - B)\subseteq t_nk+A \mbox{ for all } n\in\mathbb{N}$. $\Rightarrow \varphi_{A,k} (y^{1})\le t_n \mbox{ for all } n\in\mathbb{N}$.
Thus $\varphi_{A,k} (y^{1})\le \varphi_{A,k}(y^2)$. Hence $\varphi_{A,k} $ is $B$-monotone.
\item[(2)] Assume now \,$A-B\subseteq{\mathbb{R}}k+A=\operatorname*{dom}\varphi _{A,k}$ and that $\varphi_{A,k} $ is $B$-monotone.
Consider $a\in A$ and $b\in B$. Then $\varphi_{A,k} (a)\le 0$. Since
$a-(a-b)=b\in B$ and $a-b\in \operatorname*{dom}\varphi _{A,k}$, we obtain that $\varphi_{A,k} (a-b)\le \varphi_{A,k}
(a)\le 0$, thus $a-b\in A$ by (\ref{K4-ak}). Consequently, $A - B\subseteq A$. 
\item[(3)] Suppose $A - (B\setminus \{0\}) \subseteq \operatorname{core} A$. Take $y^1,y^2 \in F$ with $y^{2}-y^{1}\in B\setminus \{0\}$. $t:=\varphi_{A,k}(y^2)\in \mathbb{R}$. Then $y^{2}\in A+tk$ by Theorem \ref{varphi-theo-allg}(b). This implies $y^{1}\in y^{2}-(B\setminus \{0\}) \subseteq (A - (B\setminus \{0\}))+tk\subseteq \operatorname{core} A+tk$. 
By (\ref{core_in_less}), we get $\varphi_{A,k}(y^{1})<t=\varphi_{A,k}(y^{2})$. Consequently, $\varphi_{A,k}$ is strictly $B$-monotone on $F$.
\item[(4)] Assume now that $A-\mathbb{R}_{>} k\subseteq \operatorname*{core}A$, $A-B\subseteq {\mathbb{R}}k+A=\operatorname*{dom}\varphi _{A,k}$ and that $\varphi_{A,k} $ is strictly $B$-monotone.
Take $a\in A$ and $b\in B\setminus \{0\}$. Then $\varphi_{A,k} (a) \le 0$. Since
$a-(a-b)=b\in B\setminus \{0\}$ and $a-b\in \operatorname*{dom}\varphi _{A,k}$, we obtain that $\varphi_{A,k} (a-b) < \varphi_{A,k}(a)\le 0$, thus $a-b\in \operatorname*{core} A$ by (\ref{core_in_equal}). Consequently, $A-(B\setminus \{0\})\subseteq \operatorname*{core}A$.
\item[(5)] Take $y^1,y^2 \in F$ with $b^1:=y^{2}-y^{1}\in \operatorname*{core}B$. 
By the definition of the algebraic interior, there exists some $\tilde{t}\in\mathbb{R}_>$ with $b^2:=b^1-\tilde{t}k\in B$.
Hence $y^1=y^2-b^1=y^2-b^2-\tilde{t}k$.
There exists a sequence
$(t_n)_{n\in\mathbb{N}}$ of real numbers which converges to $\varphi_{A,k}(y^2)$ with $y^2\in A+t_nk$ and $t_n\geq \varphi_{A,k}(y^2)$ for all $n\in\mathbb{N}$.
Then,  for all $n\in\mathbb{N}$, $y^1\in A+t_nk-b^2-\tilde{t}k\subseteq A+(-\tilde{t}+t_n)k$ since $A - B \subseteq A$. Hence $\varphi_{A,k}(y^1)\leq -\tilde{t}+t_n$ for all $n\in\mathbb{N}$. Since the sequence $(t_n)_{n\in\mathbb{N}}$ converges to the value $\varphi_{A,k}(y^2)\in\mathbb{R}$ from above, we get $\varphi_{A,k}(y^1)\leq -\tilde{t}+\varphi_{A,k}(y^2)<\varphi_{A,k}(y^2)$.
\qed
\end{itemize}

\smallskip
The previous theorem contains some interesting special cases.

\begin{corollary}\label{t251M-recc}
\begin{itemize}
\item[]
\item[(a)] $\varphi _{A,k}$ is $(-0^+A)$-monotone.
\item[(b)] If  $\varphi _{A,k}$ is finite-valued on $F\subseteq Y$, then $\varphi _{A,k}$ is strictly $(-\operatorname{core}0^+A)$-monotone on $F$.
\item[(c)] If $A$ is a proper subset of $Y$ and $k\in -\operatorname*{core}0^+A$, then $\varphi _{A,k}$ is finite-valued and strictly $(-\operatorname{core}0^+A)$-monotone.
\end{itemize}
\end{corollary}

Part (c) of Corollary \ref{t251M-recc} results from Lemma \ref{cor_finite_Zal}.

Furthermore, Theorem \ref{t251M} and Theorem \ref{t-K4-cons} yield the following corollary.

\begin{corollary}\label{-A-mon}
Assume $A+A\subseteq A$.
\begin{itemize}
\item[(a)] $\varphi _{A,k}$ is $(-A)$-monotone.
\item[(b)] If  $\varphi _{A,k}$ is finite-valued on $F\subseteq Y$, then $\varphi _{A,k}$ is strictly $(-\operatorname{core}A)$-monotone on $F$.
\end{itemize}
\end{corollary}

The assumptions of all parts of Corollary \ref{-A-mon} do not imply that $A$ is a cone or a shifted cone, even if $A$ is a closed convex set and $\varphi_{A,k}$ is proper.
\begin{example}
In $Y=\mathbb{R}^{2}$, consider $A:=\{(y_1,y_2)^T\in\mathbb{R}^{2}\mid y_{2}\ge \frac{1}{y_{1}}, y_{1}>0\}$ and $k:=(-1,0)^T$. 
Then $k\in-0^+A$, and $A$ is a closed convex proper subset of $Y$ for that $A+A\subseteq A$ holds and that does not contain lines parallel to $k$.
\end{example}
\smallskip

\section{Functions $\varphi_{A,k}$ with varying $A$ and $k$}\label{sec-var-Ak}

Let us now investigate the influence of the choice of $k$ on the values of $\varphi _{A, k}$. In this context, we will also investigate whether the following condition is fulfilled:\\

\smallskip

\begin{tabular}{ll}
$(SP_{A,k})$: & $A$ is a proper subset of $\,Y$, $k\in -0^+A\setminus\{0\}$ and\\
&  $A$ is $k$-directionally closed. 
\end{tabular}

\begin{proposition}\label{t-scale}
Consider some arbitrary $\lambda\in\mathbb{R}_{>}$.
Then $\operatorname*{dom}\varphi _{A,\lambda k}=\operatorname*{dom}\varphi _{A,k}$ and
\begin{displaymath}
\varphi_{A,\lambda k}(y)=\frac{1}{\lambda} \varphi_{A,k}(y) \quad \mbox{ for all } y\in Y.
\end{displaymath}
$\varphi_{A,\lambda k}$ is 
proper, finite-valued, convex, concave, subadditive, superadditive, affine, linear, sublinear, positively homogeneous, odd or homogeneous if and only if $\varphi_{A,k}$ has the same property. 
If $B\subset Y$, then $\varphi_{A,\lambda k}$ is 
$B$-monotone or strictly $B$-monotone if $\varphi_{A,k}$ has the same property.\\
If $(SP_{A,k})$ is satisfied, then $(SP_{A,\lambda k})$ is fulfilled.
\end{proposition}

\begin{proof}
$\varphi _{A, \lambda k}(y)=\inf \{t\in {\mathbb{R}} \mid y\in t(\lambda k) + A\}
=\inf \{t\in {\mathbb{R}} \mid y\in (\lambda t) k + A\}=\inf \{\frac{1}{\lambda} u \mid u\in {\mathbb{R}}, y\in u k + A\}
=\frac{1}{\lambda}\inf \{u\in {\mathbb{R}} \mid y\in u k + A\}
=\frac{1}{\lambda}\varphi _{A, k}(y)\mbox{ for all } y\in Y$. 
The other assertions follow from this equation.
\end{proof}

The proposition underlines that replacing $k$ by another vector in the same direction just scales the functional. Consequently, $\varphi _{A, k}$ and $\varphi _{A, \lambda k}$, $\lambda >0$, take optimal values on some set $F\subset Y$ at the same elements of $F$. Hence it is sufficient to consider only one vector $k$ per direction in optimization problems, e.g., to restrict $k$ to unit vectors if $Y$ is a normed space.

If $\varphi _{A,k}(0)\in\mathbb{R}$, the functional can be shifted in such a way that the function value in the origin becomes zero and essential properties of the functional do not change.
\begin{proposition}\label{0-shift}
Consider some arbitrary $c\in \mathbb{R}$.
Then
\begin{eqnarray*}
\operatorname*{dom}\varphi _{A+ck,k} & = & \operatorname*{dom}\varphi _{A,k} \quad\quad\mbox{ and}\\
\varphi_{A+ck,k}(y) & = & \varphi_{A,k}(y)-c \quad\mbox{ for all } y\in Y.
\end{eqnarray*}
If $(SP_{A,k})$ is satisfied, then $(SP_{A+ck,k})$ is fulfilled.
\end{proposition}

In vector optimization or when dealing with variable domination structures, the functional is often constructed by sets that depend on some given point $y^0$.
\begin{proposition}\label{A-shift}
Consider some arbitrary $y^{0}\in Y$.
Then 
\begin{eqnarray*}
\operatorname*{dom}\varphi _{y^{0}+A,k} & = & y^{0}+\operatorname*{dom}\varphi _{A,k} \quad\mbox{ and}\\
\varphi_{y^{0}+A,k}(y) & = & \varphi_{A,k}(y-y^{0}) \quad \mbox{ for all } y\in Y.
\end{eqnarray*}
$\varphi_{y^{0}+A,k}$ is 
proper, finite-valued, convex, concave or affine if and only if $\varphi_{A,k}$ has the same property.
For $B\subset Y$, $\varphi_{y^{0}+A,k}$ is
$B$-monotone or strictly $B$-monotone if and only if $\varphi_{A,k}$ has the same property.\\
If $(SP_{A,k})$ is satisfied, then $(SP_{y^0+A,k})$ is fulfilled.
\end{proposition}
\smallskip

\section{Convex Functions with Uniform Sublevel Sets}\label{sec-cx}

In many applications, the set $A$ in the definition of the functional $\varphi_{A,k}$ is a non-trivial convex cone since it is then closely related to the cone order (cp. Section \ref{sec-Monot}). As pointed out in \cite{foe02}, for functionals $\varphi _{A,k}$ used in the formulation of risk measures, $A$ is the so-called acceptance set and just the ordering cone in a function space $L^p$. This cone has an empty interior.

Several properties of $\varphi_{A,k}$ for convex cones $A$ follow immediately from the previous sections, taking into consideration that the recession cone of a convex cone $A$ is $A$.

\begin{proposition}
Assume that $A\subset Y$ is a non-trivial convex cone. Then
\begin{itemize}
\item[(a)] $\operatorname*{dom}\varphi _{A,k}$ is convex,
\item[(b)] $\varphi_{A,k}$ is sublinear and $(-A)$-monotone.
Moreover, the function $\varphi_{A,k}$ is strictly $(-\operatorname*{core}A)$-monotone on each set $F\subseteq Y$ on that it is finite-valued.
\item[(c)] If $k\in A$, then $\varphi _{A,k}(y)=-\infty\mbox{ for all } y\in\operatorname*{dom}\varphi _{A,k}$.\\
If $k\in (-A)\cap A$, then $\operatorname*{dom}\varphi _{A,k}=A$.
\item[(d)] If $k\in -A$, then $\operatorname*{dom}\varphi _{A,k}=A+\mathbb{R}_>k$.
If, additionally, $A$ is $k$-directionally closed, then 
$$\operatorname*{lev}\nolimits_{\varphi_{A,k},\leq}(t)  =  A+tk \quad \mbox{ for all } t \in\mathbb{R}.$$
\item[(e)] $\varphi _{A,k}$ is finite-valued if and only if 
\begin{itemize}
\item[(i)] $k\in -\operatorname*{core}A$ or
\item[(ii)] $A$ is a linear subspace of $\,Y$ of codimension $1$ and $k\not\in C$.
\end{itemize} 
For $k\in -\operatorname*{core}A$, we have
$\operatorname*{lev}\nolimits_{\varphi  _{A,k} ,<}(t)=\operatorname*{core}A +tk \mbox{ for all } t \in \mathbb{R}$.\\
In case (ii), $\varphi _{A,k}$ is a linear function.
\end{itemize}
\end{proposition}
{\it Proof}
\begin{itemize}
\item[(a)] is obvious. (b) follows from Theorem \ref{t-K4-cons} and Theorem \ref{t251M}.
\item[(c)] The first statement results from Theorem \ref{varphi-theo-allg}.
$k\in (-A)\cap A$ implies $a+\mathbb{R}k\subseteq A\mbox{ for all } a\in A$ and thus $\operatorname*{dom}\varphi_{A,k}=\mathbb{R}k+A=A$.
\item[(d)] is implied by Proposition \ref{p-k-recc} and Theorem \ref{varphi-theo-allg}.
\item[(e)] If $\varphi _{A,k}$ is finite-valued, then $k\notin A$ by (c), which together with Lemma \ref{cor_finite_Zal} results in the cases (i) and (ii).\\
The assertion for case (i) follows from Theorem \ref{theo-k-core}.\\
Assume now (ii). Since each $y\in Y$ has a unique presentation $y=a+tk$ with $a\in A$ and $t\in\mathbb{R}$, $\varphi _{A,k}$ is finite-valued and linear.
\qed
\end{itemize}

\begin{lemma}\label{cone-alg-clsd}
Assume that $A\subset Y$ is a non-trivial convex cone and $k\in-\operatorname*{core}A$. Then
$A$ is $k$-directionally closed if and only if $A$ is algebraically closed.
\end{lemma}

\begin{proof}
Suppose first that $A$ is $k$-directionally closed.
Consider some arbitrary elements $a\in A$, $y\in Y$ with $a+\lambda (y-a)\in A\mbox{ for all } \lambda\in (0,1)$.
Since $Y=A+\mathbb{R}_>k$ by Lemma \ref{cone_gen_noncx}, there exist $a^0\in A$ and $t\in\mathbb{R}_>$ with $y-a=a^0+tk$.
For each $n\in\mathbb{N}_>$, $\frac{1}{n}a^0=\frac{1}{n}(y-a)-\frac{1}{n}tk$ and $y-\frac{1}{n}tk=a+(1-\frac{1}{n})(y-a)+\frac{1}{n}a^0\in A$. Hence $y\in A$ because of (\ref{k-dir-clsd}). Thus $A$ is algebraically closed. Lemma \ref{l-dir-clsd-suff} yields the assertion.
\end{proof}

In Lemma \ref{cone-alg-clsd}, the assumption $k\in-\operatorname*{core}A$ can not be replaced by $k\in (-A)\setminus A$.

\begin{example}
$A:=\{(y_1,y_2,y_3)^T\in \mathbb{R}^3 \mid y_1\geq 0,\,y_2>0,\,y_3>0\}\cup\{(y_1,y_2,y_3)^T\in \mathbb{R}^3 \mid y_1\geq 0,\,y_2=y_3=0\}$ is a convex cone, $k:=(-1,0,0)^T\in (-A)\setminus A$. $A$ is $k$-directionally closed, but $A$ is not algebraically closed.
\end{example}

\begin{corollary}\label{core-fin}
Assume that $A\subset Y$ is a non-trivial algebraically closed convex cone and $k\in  -\operatorname*{core}A$. Then $\varphi  _{A,k}$ is finite-valued, sublinear, $(-A)$-monotone, strictly $(-\operatorname*{core}A)$-monotone,
\begin{eqnarray*}
\operatorname*{lev}\nolimits_{\varphi_{A,k},\leq}(t) & = & A+tk \quad \mbox{ for all } t \in\mathbb{R},\mbox{ and}\\
\operatorname*{lev}\nolimits_{\varphi  _{A,k} ,<}(t) & = & \operatorname*{core}A +tk\quad \mbox{ for all } t \in \mathbb{R}.
\end{eqnarray*}
\end{corollary}

Functions with uniform sublevel sets that are generated by cones often coincide with a Minkowski functional on a subset of the space.

Let $p_A$ denote the Minkowski functional generated by a set $A$ in a linear space.

\begin{proposition}\label{p-Mink-uni}
Assume that $C\subset Y$ is a non-trivial algebraically closed convex cone and $k\in -\operatorname*{core}C$.
For the Minkowski functional $p_{C+k}$, we get
\[ p_{C+k}(y)=\left\{
\begin{array}{c@{\quad\mbox{ if }\quad}l}
\varphi_{C,k}(y) & y\in Y\setminus C,  \\
0 & y\in C,
\end{array}
\right.
\]
i.e., $$p_{C+k}(y)=\operatorname*{max}\{\varphi_{C,k}(y),0\}\quad\mbox{ for all } y\in Y.$$
$p_{C+k}$ is finite-valued and sublinear.
\end{proposition}

\begin{proof}
By Corollary \ref{core-fin}, $\varphi_{C,k}$ is finite-valued.\\
For each $y\in Y$, $p_{C+k}(y)=\operatorname*{inf}\{\lambda >0\mid y\in\lambda (C+k)\}=\operatorname*{inf}\{\lambda >0\mid y\in C+\lambda k\}$. Hence $p_{C+k}(y)=\varphi_{C,k}(y)$ if $\varphi_{C,k}(y)>0$. This is just the case for $y\in Y\setminus C$.\\
$C=C-\lambda k+\lambda k\subseteq C+\lambda k\mbox{ for all }\lambda >0$.
Hence $p_{C+k}(y)=0\mbox{ for all } y\in C$ and $p_{C+k}(y)=\operatorname*{max}\{\varphi_{C,k}(y),0\}\mbox{ for all } y\in Y$.
Since $C+k$ is convex and absorbing, $p_{C+k}$ is sublinear by \cite[Lemma 5.50]{AliBor06}.
\end{proof}

We are now going to investigate the relationship between functions with uniform sublevel sets and norms that are defined by the Minkowski functional of an order interval.
Jahn proved the following statement \cite[Lemma 1.45]{Jah86a}.

\begin{lemma}\label{p-norm-ordint}
Suppose that $C \subset Y$ is a non-trivial algebraically closed convex pointed cone and $k\in\operatorname*{core}C$. Then the Minkowski functional of the order interval $[-k,k]_C$ is a norm.
\end{lemma}

\begin{remark}
Let $C$ be an ordering cone in $Y$. Then it is obvious, that $k\in Y$ is an order unit of $Y$ if and only if $k\in\operatorname*{core}C$. The order unit norm, which is often denoted by $\| \cdot \|_{\infty}$, is just the norm constructed in Lemma \ref{p-norm-ordint}. For details related to order units, see \cite{AliBor06} and \cite{AliTou07}.
\end{remark}

\begin{proposition}\label{p-ordint-varphi}
Suppose that $C \subset Y$ is a non-trivial algebraically closed convex pointed cone with $k\in\operatorname*{core}C$, $a\in Y$. Denote by $\|\cdot \|_{C,k}$ the norm that is given as the Minkowski functional of the order interval $[-k,k]_C$. Then
\[ \label{eq-norm-varphi}
\| y-a\|_{C,k}=\varphi_{a-C,k}(y)\quad\mbox{ for all } y\in a+C.
\]
\end{proposition}

\begin{proof}
Consider some $y\in a+C$.
\begin{eqnarray*}
\| y-a\| _{C,k} & = & \operatorname*{inf}\{\lambda >0\mid y-a\in \lambda ((C-k)\cap (k-C))\}\\
& = & \operatorname*{inf}\{\lambda >0\mid y-a\in  (C-\lambda k)\cap (\lambda k-C)\}\\
& = & \operatorname*{inf}\{\lambda >0\mid y-a\in  \lambda k-C\} \mbox{ since } y-a\in C\subseteq C-\lambda k\mbox{ for all }\lambda\in\mathbb{R}_+.\\
& = & \varphi_{a-C,k}(y) \mbox{ if } y\notin a-C.
\end{eqnarray*}
$(a+C)\cap (a-C)=\{a\}$ since $C$ is pointed. Hence $\| y-a\| _{C,k}=\varphi_{a-C,k}(y)\mbox{ for all } y\in a+C$ with $y\not=a$. $\| a-a\| _{C,k}=0=\varphi_{a-C,k}(a)$. 
\end{proof}

In many applications, solutions are determined by problems $\operatorname*{min}_{y\in F}\| y-a\| _{C,k}$ with $F\subseteq a+C$. Replacing $\| y-a\| _{C,k}$ by $\varphi_{a-C,k}(y)$, this approach can often be applied without the assumption $F\subseteq a+C$. This is illustrated for the scalarization of vector optimization problems with the weighted Chebyshev norm and with extensions of this norm in \cite{Wei90} and \cite{Wei94}.
\smallskip

\section{Decision Making and Vector Optimization}\label{s-decision}

Consider the following general decision problem:\\
A decision maker (DM) wants to make a decision by choosing an element from a set $S$ of feasible decisions, where the outcomes of the decisions are given by some function $f:S\to Y$.

What is a best decision depends on the DM's preferences in the set $F:=f(S)$ of outcomes.  
Let $\succ$ denote the DM's strict or weak preference relation on $Y$. Then the set of decision outcomes that are optimal for the DM is just $\operatorname*{Min}(F,\succ ):=\{y^0\in F\mid \forall y\in F: (y\succ y^0\Rightarrow y^0\succ y)\}$.

Note that $\succ$ consists of the preferences the DM is aware of. This relation is refined during the decision process, but fixed in each single step of the decision process, where information about $\operatorname*{Min}(F,\succ )$ should support the DM in formulating further preferences. In the final phase of the decision process, the DM chooses one decision, but in the previous phases $\operatorname*{Min}(F,\succ )$ contains more than one element.

\begin{definition}
Suppose $\succ$ to be a binary relation on $Y$.\\
$d\in Y$ is said to be a domination factor of $y\in Y$ if $y\succ y+d$.
We define the domination structure of $\succ$ by $D_{\succ }:Y\to {\mathcal{P}}(Y)$ with $D_{\succ }(y):=\{d\in Y\mid y\succ y+d\}$ for each $y\in Y$.
If there exists some set $D\subseteq Y$ with $D_{\succ }(y)=D$ for all $y\in Y$, then $D$ is called the domination set of $\succ$.
\end{definition}

The definition implies:

\begin{proposition}
Suppose $\succ$ to be a binary relation on $Y$, $D\subseteq Y$.\\
$D$ is a domination set of $\succ$ if and only if:
$$\forall y^1,y^2\in Y:\;\; (y^1\succ y^2\iff y^2\in y^1+D).$$
There exists a domination set of $\succ$ if and only if:
\begin{equation}\label{domset-add}
\forall y^1,y^2,y\in Y:\;\; (y^1\succ y^2\implies (y^1+y)\succ (y^2+y)).
\end{equation}
If $D$ is a domination set of $\succ$, we have:
\begin{itemize}
\item[(a)] $\succ$ is reflexive $\iff 0\in D$.
\item[(b)] $\succ$ is asymmetric $\iff D\cap (-D)=\emptyset$.
\item[(c)] $\succ$ is antisymmetric $\iff D\cap (-D)=\{0\}$.
\item[(d)] $\succ$ is transitive $\iff D+D\subseteq D$.
\item[(e)] $\succ$ fulfills the condition
\begin{equation}
\forall\, y^{1},y^{2}\in Y \;\; \forall\, \lambda \in {\mathbb{R}_>}:
\;\; y^{1} \succ y^{2} \implies (\lambda y^{1})\succ (\lambda y^{2}), \label{f217-y} 
\end{equation}
if and only if $D\cup\{0\}$ is a cone.
\item[(f)] $\succ$ is a transitive relation that fulfills condition (\ref{f217-y}) if and only if $D\cup\{0\}$ is a convex cone.
\item[(g)] $\succ$ is a partial order that fulfills condition (\ref{f217-y}) if and only if $D$ is a pointed convex cone.
\end{itemize}
\end{proposition}

\begin{example}
Not each preference relation fulfills the conditions (\ref{domset-add}) and (\ref{f217-y}). Consider $y$ to be the number of tea spoons full of sugar that a person puts into his coffee. He could prefer $2$ to $1$, but possibly not $2+2$ to $1+2$ or also not prefer $2\times 2$ to $2\times 1$.
\end{example}

We now introduce optimal elements w.r.t. sets as a tool for finding optimal elements w.r.t. relations. 

\begin{definition}
Suppose $F, D\subseteq Y$.
An element $y^{0}\in F$ is called an efficient element of $F$
w.r.t. $D$ iff 
\[
F\cap(y^0-D)\subseteq \{y^0\}.
\]
We denote the set of efficient elements of $F$ w.r.t. $D$ by $\operatorname*{Eff}(F,D).$
\end{definition}

We get \cite[p.51]{Wei85}:

\begin{proposition}\label{p-domstruct-D}
Suppose $\succ$ to be a (not necessarily strict) preference relation on $Y$ with domination structure $D{_\succ }$, $D\subseteq Y$.
\begin{itemize}
\item[(a)] If $D{_\succ }(y)=D\mbox{ for all } y\in F$, then $\operatorname*{Min}(F,\succ)=\operatorname*{Eff}(F,D\setminus (-D))$.\\
If $\succ$ is asymmetric or antisymmetric, then $\operatorname*{Min}(F,\succ)=\operatorname*{Eff}(F,D)$. 
\item[(b)] If $D{_\succ }(y)\subseteq D\mbox{ for all } y\in F$, then $\operatorname*{Eff}(F,D)\subseteq\operatorname*{Min}(F,\succ)$.
\item[(c)] If $\succ$ is asymmetric or antisymmetric and $D\subseteq D{_\succ }(y)\mbox{ for all } y\in F$, then\\
$\operatorname*{Min}(F,\succ)\subseteq \operatorname*{Eff}(F,D)$.
\end{itemize}
\end{proposition}
{\it Proof}
\begin{itemize}
\item[(a)] Consider some $y^0\in F$.
\begin{eqnarray*}
y^0\notin \operatorname*{Min}(F,\succ) & \Leftrightarrow & \exists y\in F:\;y\succ y^0,\mbox{ but } \neg (y^0\succ y),\\
& \Leftrightarrow & \exists y\in F:\;y^0\in y+D,\mbox{ but } y\notin y^0+D,\\
& \Leftrightarrow & \exists y\in F:\;y^0\in y+(D\setminus (-D))\\
& \Leftrightarrow & y^0\notin \operatorname*{Eff}(F,D\setminus (-D)).
\end{eqnarray*}
If $\succ$ is asymmetric or antisymmetric, then $D\cap(-D)\subseteq\{0\}$. $\Rightarrow D\setminus (-D)=D$ or $D\setminus (-D)=D\setminus\{0\}$. $\Rightarrow \operatorname*{Eff}(F,D\setminus (-D))=\operatorname*{Eff}(F,D)$.
\item[(b)] Consider some $y^0\in F\setminus \operatorname*{Min}(F,\succ)$.
$\Rightarrow \exists y\in F:\;y\succ y^0,\mbox{ but } \neg (y^0\succ y)$.
$\Rightarrow \exists y\in F\setminus\{y^0\}:\; y^0\in y+D{_\succ }(y)$.
$\Rightarrow \exists y\in F\setminus\{y^0\}:\; y\in y^0-D{_\succ }(y)\subseteq y^0-D$. 
$\Rightarrow y^0\notin \operatorname*{Eff}(F,D)$.
\item[(c)] Consider some $y^0\in F\setminus \operatorname*{Eff}(F,D)$.
$\Rightarrow \exists y\in F\setminus\{y^0\}:\; y\in y^0-D\subseteq y^0-D{_\succ }(y)$. 
$\Rightarrow \exists y\in F\setminus\{y^0\}:\; y^0\in y+D{_\succ }(y)$.
$\Rightarrow \exists y\in F\setminus\{y^0\}:\;y\succ y^0$.
$\Rightarrow y^0\notin \operatorname*{Min}(F,\succ)$, since $\succ$ is asymmetric or antisymmetric.
\qed
\end{itemize}

\begin{remark}
In \cite{Wei83} and \cite{Wei85}, optimal elements of sets w.r.t. relations and sets were investigated under the general assumptions given here. There exist earlier papers that study optima w.r.t. quasi orders, e.g. \cite{Wall45} and \cite{Ward54}, or optimal elements w.r.t. ordering cones, e.g. \cite{hurw58} and \cite{yu73dom}. The concept of domination structures goes back to Yu \cite{yu73dom}. Domination factors according to the above definition were introduced in \cite{Berg76},
where minimal elements w.r.t. convex sets $D$ with $0\in D\setminus\operatorname*{int}D$ were investigated in $\mathbb{R}^\ell$.   
\end{remark}

The domination factors refer to elements that are dominated. Of course, a structure could also be built by dominating elements. Such a structure was studied by Chen \cite{Chen92} and later in the books by himself et al. \cite{CheHuaYan05} for the case that the structure consists of convex cones or of convex sets that contain zero in their boundary.

\begin{definition}
Suppose $\succ$ to be a binary relation on $Y$.\\
$\tilde{d}\in Y$ is said to be a pre-domination factor of $y\in Y$ if $\;y-\tilde{d}\succ y$.
We define the pre-domination structure of $\succ$ by $\tilde{D}_{\succ } : Y\to {\mathcal{P}}(Y)$ with $\tilde{D}_{\succ }(y):=\{\tilde{d}\in Y\mid y-\tilde{d}\succ y\}$ for each $y\in Y$.
\end{definition}

A pre-domination structure is constant on the entire space if and only if the domination structure of the same relation is constant on the whole space.

\begin{proposition}
Suppose $\succ$ to be a (not necessarily strict) preference relation on $Y$ with pre-domination structure $\tilde{D}_{\succ }$.
There exists a domination set $D\subseteq Y$ of $\succ$ if and only if $\tilde{D}_{\succ }(y)=D\mbox{ for all } y\in Y$.
\end{proposition} 

\begin{proof}
Let $D_{\succ }$ denote the domination structure of $\succ$.\\
$D_{\succ }(y)=D$ holds for all $y\in Y$ if and only if:\\
$\forall y\in Y:\;(y\succ y+d\Leftrightarrow d\in D)$, i.e., if and only if\\
$\forall y\in Y:\;(y-d\succ y\Leftrightarrow d\in D)$, which is equivalent to $\tilde{D}_{\succ }(y)=D\mbox{ for all } y\in Y$.
\end{proof}

The pre-domination structure may consist of convex sets when this is not the case for the domination structure.

\begin{example}\label{ex-predom}
Define on $Y=\mathbb{R}^2$ the relation $\succ$ by: $y^1\succ y^2\Leftrightarrow \| y^1\| _2\leq \| y^2\| _2$.
The pre-domination structure, but not the domination structure,  consists of convex sets. 
\end{example}

Analogously to Proposition \ref{p-domstruct-D}, the following relationships between pre-domination structures and optima w.r.t. sets hold.

\begin{proposition}
Suppose $\succ$ to be a (not necessarily strict) preference relation on $Y$ with pre-domination structure $\tilde{D}_{\succ }$, $D\subseteq Y$.
\begin{itemize}
\item[(a)] If $\tilde{D}_{\succ }(y)=D\mbox{ for all } y\in F$, then $\operatorname*{Min}(F,\succ)=\operatorname*{Eff}(F,D\setminus (-D))$.\\
If $\succ$ is asymmetric or antisymmetric, then $\operatorname*{Min}(F,\succ)=\operatorname*{Eff}(F,D)$. 
\item[(b)] If $\tilde{D}_{\succ }(y)\subseteq D\mbox{ for all } y\in F$, then $\operatorname*{Eff}(F,D)\subseteq\operatorname*{Min}(F,\succ)$.
\item[(c)] If $\succ$ is asymmetric or antisymmetric and $D\subseteq \tilde{D}_{\succ }(y)\mbox{ for all } y\in F$, then\\ $\operatorname*{Min}(F,\succ)\subseteq \operatorname*{Eff}(F,D)$.
\end{itemize}
\end{proposition}
{\it Proof}
\begin{itemize}
\item[(a)] Consider some $y^0\in F$.
\begin{eqnarray*}
y^0\notin \operatorname*{Min}(F,\succ) & \Leftrightarrow & \exists y\in F:\;y\succ y^0,\mbox{ but } \neg (y^0\succ y),\\
& \Leftrightarrow & \exists y\in F:\;y\in y^0-D,\mbox{ but } y^0\notin y-D,\\
& \Leftrightarrow & \exists y\in F:\;y^0\in y+(D\setminus (-D))\\
& \Leftrightarrow & y^0\notin \operatorname*{Eff}(F,D\setminus (-D)).
\end{eqnarray*}
If $\succ$ is asymmetric or antisymmetric, then $D\cap(-D)\subseteq\{0\}$. $\Rightarrow D\setminus (-D)=D\setminus\{0\}$. $\Rightarrow \operatorname*{Eff}(F,D\setminus (-D))=\operatorname*{Eff}(F,D)$.
\item[(b)] Consider some $y^0\in F\setminus \operatorname*{Min}(F,\succ)$.
$\Rightarrow \exists y\in F:\;y\succ y^0,\mbox{ but } \neg (y^0\succ y)$.
$\Rightarrow \exists y\in F\setminus\{y^0\}:\; y\in y^0-\tilde{D}_{\succ }(y^0)\subseteq y^0-D$.
$\Rightarrow y^0\notin \operatorname*{Eff}(F,D)$.
\item[(c)] Consider some $y^0\in F\setminus \operatorname*{Eff}(F,D)$.
$\Rightarrow \exists y\in F\setminus\{y^0\}:\; y\in y^0-D\subseteq y^0-\tilde{D}_{\succ }(y^0)$. 
$\Rightarrow \exists y\in F\setminus\{y^0\}:\;y\succ y^0$.
$\Rightarrow y^0\notin \operatorname*{Min}(F,\succ)$, since $\succ$ is asymmetric or antisymmetric.
\qed
\end{itemize}

\smallskip
Since the domination structure as well as the pre-domination structure completely characterize the binary relation, the minimal point set w.r.t. the relation can be described via the domination structure or via the pre-domination structure. 

\begin{lemma}\label{l-minrel-mindom}
Suppose $\succ$ to be a binary relation on $Y$ with domination structure $D_{\succ }$ and pre-domination structure $\tilde{D}_{\succ }$, $F\subseteq Y$.
\begin{eqnarray*}
\operatorname*{Min}(F,\succ) &= &\{y^0\in F\mid \forall y\in F:\;(y^0\in y+D_{\succ }(y)\Rightarrow y\in y^0+D_{\succ }(y^0))\}\\
& = & \{y^0\in F\mid \forall y\in F:\;(y\in y^0-\tilde{D}_{\succ }(y^0)\Rightarrow y^0\in y-\tilde{D}_{\succ }(y))\}.
\end{eqnarray*}
If $\succ$ is asymmetric or antisymmetric, we get
\begin{eqnarray*}
\operatorname*{Min}(F,\succ) &= &\{y^0\in F\mid \not\exists y\in F\setminus\{y^0\}:\;y^0\in y+D_{\succ }(y)\}\\
& = & \{y^0\in F\mid \not\exists y\in F\setminus\{y^0\}:\;y\in y^0-\tilde{D}_{\succ }(y^0)\}.
\end{eqnarray*}
\end{lemma}

In the case that the domination structure can be described by a domination set, the decision problem becomes a vector optimization problem, which we will study in the next sections.
\smallskip

\section{Basic Properties of the Efficient and the Weakly Efficient Point Set in Vector Optimization}\label{s-basics-vo}

In this section, we will define the vector optimization problem and prove some basic properties of its solutions. We will show in which way minimal solutions of scalar-valued functions can deliver solutions to the vector optimization problem.

The vector optimization problem is given by a function $f:S\rightarrow Y$, mapping a set $S$ into $Y$, and a subset $D\not= Y$ of $\,Y$ that defines the solution concept. A solution of the vector optimization problem is each $s\in S$ with $f(s)\in \operatorname*{Eff}(f(S),D)$. 

Hence we are interested in the efficient elements of $F:=f(S)$ w.r.t. $D$. One can imagine that for each $y^0\in F$ the set of elements in $F$ that is preferred to $y^0$ is just $F\cap (y^0-(D\setminus\{0\}))$.
We will call $D$ the domination set of the vector optimization problem.

\begin{remark}
Weidner (\cite{Wei83}, \cite{Wei85}, \cite{Wei90}) studied vector optimization problems under such general assumptions motivated by decision theory. Here, we only refer to a part of those results. If $D$ is an ordering cone in $Y$, $\operatorname*{Eff}(F,D)$ is the set of elements of $F$ that are minimal w.r.t. the cone order $\leq_D$. In the literature, vector optimization problems are usually defined with domination sets that are ordering cones.
\end{remark}

It turns out that, in general, it is easier to determine efficient elements w.r.t. the core of sets.
\begin{definition}
$\operatorname*{WEff}(F,D):=\operatorname*{Eff}(F,\operatorname*{core}D)$ is said to be the set of 
weakly efficient elements of $F$ w.r.t. $D$.
\end{definition}

Here, weak efficiency is introduced using the algebraic interior (cp. \cite{Jah86a}), since we do not want to assume that $D$ has a nonempty interior in some topological vector space. In topological vector spaces, weak efficiency is usually defined with the topological interior instead of the core.

We will first show some basic properties of efficient and weakly efficient point sets. These statements include relationships between the efficient point set of $F$ and the efficient point set of $F+D$. In applications, $F+D$ may be algebraically closed or convex though $F$ does not have this property. 

\begin{lemma}\label{l-Eff-prop}
\begin{itemize}
\item[]
\item[(a)] $\operatorname*{Eff}(F,D)=\operatorname*{Eff}(F,D\cup\{0\})=\operatorname*{Eff}(F,D\setminus\{0\})$.
\item[(b)] $D_1\subseteq D\implies \operatorname*{Eff}(F,D)\subseteq \operatorname*{Eff}(F,D_1)$.
\item[(c)] $F_1\subseteq F\implies \operatorname*{Eff}(F,D)\cap F_1\subseteq \operatorname*{Eff}(F_1,D)$.
\item[(d)] Suppose $F\subseteq A\subseteq F+(D\cup\{0\})$.
Then $\operatorname*{Eff}(A,D)\subseteq \operatorname*{Eff}(F,D)$.\\
If, additionally, $D\cap (-D)\subseteq \{0\}$ and $D+D\subseteq D$, then\\
$\operatorname*{Eff}(A,D)=\operatorname*{Eff}(F,D)$.
\item[(e)] If $D+D\subseteq D$, then $\operatorname*{Eff}(F\cap (y-D),D)=\operatorname*{Eff}(F,D)\cap (y-D)\mbox{ for all } y\in Y$.
\end{itemize}
\end{lemma}
{\it Proof}
\begin{itemize}
\item[(a)]-(c) follow immediately from the definition of efficient elements.
\item[(d)] Consider some $y^0\in \operatorname*{Eff}(A,D)$.
Since $A\subseteq F+(D\cup\{0\})$, there exist $y^1\in F$, $d\in D\cup\{0\}$ with $y^0=y^1+d$.
$\Rightarrow y^1\in y^0-(D\cup\{0\})$. $\Rightarrow y^1=y^0$ because of $y^0\in \operatorname*{Eff}(A,D)$ and $F\subseteq A$.
$\Rightarrow y^0\in F$. Hence $\operatorname*{Eff}(A,D)\subseteq F$.
$\operatorname*{Eff}(A,D)\subseteq \operatorname*{Eff}(F,D)$ follows from (c), since $F\subseteq A$.\\
Assume now $D\cap (-D)\subseteq \{0\}$ and $D+D\subseteq D$. Suppose that $\operatorname*{Eff}(A,D)\not=\operatorname*{Eff}(F,D)$. $\Rightarrow \exists y^0\in\operatorname*{Eff}(F,D)\setminus\operatorname*{Eff}(A,D)$.
$\Rightarrow \exists a\in A:\;a\in y^0-(D\setminus\{0\})$.
$ A\subseteq F+(D\cup\{0\}) \Rightarrow \exists y\in F, d\in D\cup\{0\} :\;y+d\in y^0-(D\setminus\{0\})$.
$\Rightarrow y^0-y\in d+(D\setminus\{0\})\subseteq D$. $\Rightarrow y^0=y$, since $y^0\in\operatorname*{Eff}(F,D)$.
Then $y+d\in y^0-(D\setminus\{0\})$ implies $d\in -(D\setminus\{0\})$, a contradiction to $D\cap (-D)\subseteq \{0\}$. 
\item[(e)] Choose some $y\in Y$.
$\operatorname*{Eff}(F\cap (y-D),D)\supseteq\operatorname*{Eff}(F,D)\cap (y-D)$ results from (c).
Assume there exists some $y^0\in\operatorname*{Eff}(F\cap (y-D),D)\setminus\operatorname*{Eff}(F,D)$.
$\Rightarrow (y^0-D)\cap (F\cap (y-D))\subseteq \{y^0\}$ and $\exists y^1\not=y^0:\; y^1\in F\cap(y^0-D)$.
$y^1\in y^0-D\subseteq (y-D)-D\subseteq y-D$. $\Rightarrow y^1\in (y^0-D)\cap (F\cap (y-D))$, a contradiction.
\qed
\end{itemize}

\begin{remark}
The statements of Lemma \ref{l-Eff-prop} were proved in \cite{Wei83} and \cite{Wei85}.
$\operatorname*{Eff}(F+D,D)\subseteq \operatorname*{Eff}(F,D)$ was shown before in $Y=\mathbb{R}^\ell$ by Bergstresser et al. \cite[Lemma 2.2]{Berg76} for convex sets $D$ with $0\in D\setminus \operatorname*{int}D$, by Sawaragi, Nakayama and Tanino 
\cite[Prop. 3.1.2]{sana85} for cones $D$. Vogel \cite[Satz 6]{Vog75} proved
$\operatorname*{Eff}(F+D,D)= \operatorname*{Eff}(F,D)$ for pointed convex cones $D$.
\end{remark}

Yu \cite[p.20]{yu74} gave an example for a convex cone that is not pointed and for that $\operatorname*{Eff}(F,D)$ is not a subset of $\operatorname*{Eff}(F+D,D)$.

We get for weakly efficient elements:

\begin{lemma}\label{l-wEff-prop}
\begin{itemize}
\item[]
\item[(a)] $\operatorname*{Eff}(F,D)\subseteq \operatorname*{WEff}(F,D)$.
\item[(b)] $\operatorname*{core}D_1\subseteq \operatorname*{core}D\implies \operatorname*{WEff}(F,D)\subseteq \operatorname*{WEff}(F,D_1)$.
\item[(c)] $F_1\subseteq F\implies \operatorname*{WEff}(F,D)\cap F_1\subseteq \operatorname*{WEff}(F_1,D)$.
\item[(d)] Suppose $F\subseteq A\subseteq F\cup (F+\operatorname*{core}D)$.
Then $\operatorname*{WEff}(A,D)\subseteq \operatorname*{WEff}(F,D)$.\\
If, additionally, $\operatorname*{core}D\cap (-\operatorname*{core}D)\subseteq \{0\}$ and $\operatorname*{core}D+\operatorname*{core}D\subseteq \operatorname*{core}D$, then
$\operatorname*{WEff}(A,D)=\operatorname*{WEff}(F,D)$.
\item[(e)] Suppose $F\subseteq A\subseteq F+(D\cup\{0\})$, $0\notin\operatorname*{core}D$ and $D+\operatorname*{core}D\subseteq \operatorname*{core}D$.\\
Then $\operatorname*{WEff}(A,D)\cap F=\operatorname*{WEff}(F,D)$.
\item[(f)] If $D+\operatorname*{core}D\subseteq D$, then\\
$\operatorname*{WEff}(F\cap (y-D),D)=\operatorname*{WEff}(F,D)\cap (y-D)\mbox{ for all } y\in Y$.
\end{itemize}
\end{lemma}
{\it Proof}
\begin{itemize}
\item[(a)] follows from Lemma \ref{l-Eff-prop}(b) with $D_1=\operatorname*{core}D$.
\item[(b)]- (d) result from Lemma \ref{l-Eff-prop}(b)-(d) when replacing $D$ and $D_1$ by their core.
\item[(e)] Consider some $y^0\in F$ with $y^0\notin \operatorname*{WEff}(A,D)$. $\Rightarrow \exists a\in A\setminus\{y^0\}:\;\; y^0\in a+\operatorname*{core}D$.
Since $A\subseteq F+(D\cup\{0\})$, there exist $y^1\in F, d\in D\cup\{0\}$ such that $y^0\in y^1+d+\operatorname*{core}D\subseteq y^1+\operatorname*{core}D$.
$\Rightarrow y^0\notin \operatorname*{WEff}(F,D)$.\\
Hence $\operatorname*{WEff}(F,D)\subseteq \operatorname*{WEff}(A,D)$. The assertion follows by (c) since $F\subseteq A$.
\item[(f)] Choose some $y\in Y$.
$\operatorname*{WEff}(F\cap (y-D),D)\supseteq\operatorname*{WEff}(F,D)\cap (y-D)$ results from (c).
Assume there exists some $y^0\in\operatorname*{WEff}(F\cap (y-D),D)\setminus\operatorname*{WEff}(F,D)$.
$\Rightarrow (y^0-\operatorname*{core}D)\cap (F\cap (y-D))\subseteq \{y^0\}$ and $\exists y^1\not=y^0:\; y^1\in F\cap(y^0-\operatorname*{core}D)$.
$y^1\in y^0-\operatorname*{core}D\subseteq (y-D)-\operatorname*{core}D\subseteq y-D$. $\Rightarrow y^1\in (y^0-\operatorname*{core}D)\cap (F\cap (y-D))$, a contradiction.
\qed
\end{itemize}

\begin{corollary}\label{c-FplusD}
Assume $D$ is a non-trivial pointed convex cone and $F\subseteq A\subseteq F+D$.
\begin{itemize}
\item[(a)] $\operatorname*{Eff}(A,D)=\operatorname*{Eff}(F,D)$.
\item[(b)] $\operatorname*{WEff}(A,D)\cap F=\operatorname*{WEff}(F,D)$.
\end{itemize}
\end{corollary}

Podinovskij and Nogin \cite[Lemma 2.2.1]{pono82} proved part (b) of Corollary \ref{c-FplusD} in $Y=\mathbb{R}^\ell$ for $D=\mathbb{R}^\ell _+$ and $A=F+\mathbb{R}^\ell _+$. They gave the following example that, in general, $\operatorname*{WEff}(F+\mathbb{R}^\ell _+,D)=\operatorname*{WEff}(F,\mathbb{R}^\ell _+)$ is not fulfilled.

\begin{example}
$Y:=\mathbb{R}^2$, $F:=\{(y_1,y_2)\in Y\mid y_2>0\}\cup\{0\}$.
Then $F+\mathbb{R}^2_+=F\cup\{(y_1,y_2)\in Y\mid y_1>0, y_2=0\}$.
$(1,0)^T\in \operatorname*{WEff}(F+\mathbb{R}^2_+,\mathbb{R}^2_+)\setminus \operatorname*{WEff}(F,\mathbb{R}^2_+)$, since $(1,0)^T\notin F$.
\end{example}

Efficient elements are usually not located in the core of the feasible point set.

\begin{proposition}\label{eff-in-algcl}
Assume that there exists some $d\in D\setminus\{0\}$ such that $td\in D\setminus\{0\}\mbox{ for all } t\in (0,1)$. Then
$$\operatorname*{Eff}(F,D)\subseteq F\setminus\operatorname*{core}F.$$
\end{proposition}
\begin{proof}
Consider some $y^0\in \operatorname*{core}F$.
$\Rightarrow \exists t\in\mathbb{R}_>:\; t<1$ and $y^0-td\in F$. $\Rightarrow y^0\notin \operatorname*{Eff}(F,D)$. 
\end{proof}

The assumption is fulfilled, if $D\cup\{0\}$ is star-shaped about zero. This is the case if $D$ is a cone or $D\cup\{0\}$ is convex.

\begin{theorem}\label{weff-nocore}
\begin{itemize}
\item[]
\item[(a)] We have 
\begin{equation}\label{core-in}
F\setminus\operatorname*{core}(F+D)\subseteq\operatorname*{WEff}(F,D).
\end{equation}
\item[(b)] If $D$ is a convex cone, then
$$\operatorname*{WEff}(F,D)=F\setminus\operatorname*{core}(F+D).$$
\end{itemize}
\end{theorem}
{\it Proof}
\begin{itemize}
\item[(a)]If $y^0\in F\setminus\operatorname*{WEff}(F,D)$, then there exists some $y\in F\cap(y^0-\operatorname*{core}D)$, which implies $y^0\in\operatorname*{core}(F+D)$. Thus (\ref{core-in}) holds.
\item[(b)] $\operatorname*{WEff}(F,D)=\operatorname*{WEff}(F+D,D)\cap F$ by Lemma \ref{l-wEff-prop}. $\operatorname*{WEff}(F+D,D)\subseteq (F+D)\setminus\operatorname*{core}(F+D)$ by Proposition \ref{eff-in-algcl}.
\qed
\end{itemize}

Let us now formulate sufficient conditions for solutions of vector optimization problems by minima of scalar-valued functions.
We will use the abbreviation $\operatorname*{argmin}\nolimits_{M}\varphi :=\operatorname*{argmin}\nolimits_{y\in M}\varphi (y)$.
One can show \cite{Wei90}:

\begin{proposition}\label{p-mon-scal}
Assume $\varphi :F\to \mathbb{R}$.
\begin{itemize}
\item[(a)] $\operatorname*{Eff}(F,D)\cap \operatorname*{argmin}\nolimits_{F}\varphi \subseteq \operatorname*{Eff}(\operatorname*{argmin}\nolimits_{F}\varphi ,D)$.
\item[(b)] If $\varphi $ is $D$--monotone on $F$, then $\operatorname*{Eff}(F,D)\cap \operatorname*{argmin}\nolimits_{F}\varphi = \operatorname*{Eff}(\operatorname*{argmin}\nolimits_{F}\varphi ,D)$.\\
If, additionally, $\operatorname*{argmin}\nolimits_{F}\varphi =\{y^0\}$, then $y^0\in \operatorname*{Eff}(F,D)$.
\item[(c)] $\operatorname*{argmin}\nolimits_{F}\varphi \subseteq \operatorname*{Eff}(F,D)$ holds, if $\varphi$ is strictly $D$--monotone on $F$.
\end{itemize}
\end{proposition}
{\it Proof}
\begin{itemize}
\item[(a)] results from Lemma \ref{l-Eff-prop}(c), since $\operatorname*{argmin}\nolimits_{F}\varphi \subseteq F$.
\item[(b)] Consider some $y^0\in \operatorname*{Eff}(\operatorname*{argmin}\nolimits_{F}\varphi ,D)$ and assume\\
$y^0\notin \operatorname*{Eff}(F,D)\cap \operatorname*{argmin}\nolimits_{F}\varphi $.
$\Rightarrow y^0\notin \operatorname*{Eff}(F,D)$. $\Rightarrow \exists y\not= y^0:\; y\in F\cap (y^0-D)$.
$\Rightarrow y^0-y\in D$. $\Rightarrow \varphi (y^0)\geq \varphi (y)$, since $\varphi $ is $D$--monotone on $F$.
$\Rightarrow y\in \operatorname*{argmin}\nolimits_{F}\varphi \cap (y^0-D)$. $\Rightarrow y^0\notin \operatorname*{Eff}(\operatorname*{argmin}\nolimits_{F}\varphi ,D)$, a contradiction.\\
If $\operatorname*{argmin}\nolimits_{F}\varphi =\{y^0\}$, then $\operatorname*{Eff}(\operatorname*{argmin}\nolimits_{F}\varphi ,D)=\{y^0\}$, which yields the assertion.
\item[(c)] Consider some $y^0\in \operatorname*{argmin}\nolimits_{F}\varphi $ and assume $y^0\notin \operatorname*{Eff}(F,D)$.
$\Rightarrow \exists y\not= y^0:\; y\in F\cap (y^0-D)$.
$\Rightarrow y^0-y\in D\setminus\{0\}$. $\Rightarrow \varphi (y^0)> \varphi (y)$, since $\varphi $ is strictly $D$--monotone on $F$, a contradiction to $y^0\in \operatorname*{argmin}\nolimits_{F}\varphi $.
\qed
\end{itemize}

\smallskip
Immediately from the previous proposition, we get the related statements for weakly efficient elements.

\begin{proposition}\label{p-mon-wscal}
Assume $\varphi :F\to \mathbb{R}$.
\begin{itemize}
\item[(a)] $\operatorname*{WEff}(F,D)\cap \operatorname*{argmin}\nolimits_{F}\varphi \subseteq \operatorname*{WEff}(\operatorname*{argmin}\nolimits_{F}\varphi ,D)$.
\item[(b)] If $\varphi $ is $(\operatorname*{core}D)$-monotone on $F$, then\\
$\operatorname*{WEff}(F,D)\cap \operatorname*{argmin}\nolimits_{F}\varphi = \operatorname*{WEff}(\operatorname*{argmin}\nolimits_{F}\varphi ,D)$.\\
If, additionally, $\operatorname*{argmin}\nolimits_{F}\varphi =\{y^0\}$, then $y^0\in \operatorname*{WEff}(F,D)$.
\item[(c)] $\operatorname*{argmin}\nolimits_{F}\varphi \subseteq \operatorname*{WEff}(F,D)$ holds if $\varphi$ is strictly $(\operatorname*{core}D)$-monotone on $F$.
\end{itemize}
\end{proposition}

\smallskip

\section{Scalarization in Vector Optimization by Functionals with Uniform Sublevel Sets}\label{s-scal-uni}

We will now derive conditions for efficient and weakly efficient elements by functionals with uniform sublevel sets. Here, we use functionals $\varphi _{a-H,k}$, where $a\in Y$ can be considered to be some reference point and $H\subset Y$ is a set related to the domination set $D$.

We assume that $H\not=\{0\}$ is a proper subset of $Y$ and $a\in Y$. 

Even if the functionals $\varphi _{a-H,k}$ are not defined on the whole set $F$, they can deliver efficient and weakly efficient elements of $F$.

\begin{lemma}\label{hab-l421}
\begin{itemize}
\item[]
\item[(a)] $\operatorname*{Eff}(F,D)\cap \operatorname*{dom}\varphi _{a-H,k}\subseteq \operatorname*{Eff}(F\cap\operatorname*{dom}\varphi _{a-H,k},D)$.
\item[(b)] $H+D\subseteq H\implies  \operatorname*{Eff}(F,D)\cap \operatorname*{dom}\varphi _{a-H,k}= \operatorname*{Eff}(F\cap\operatorname*{dom}\varphi _{a-H,k},D)$.
\item[(c)] $\operatorname*{WEff}(F,D)\cap \operatorname*{dom}\varphi _{a-H,k}\subseteq \operatorname*{WEff}(F\cap\operatorname*{dom}\varphi _{a-H,k},D)$.
\item[(d)] $H+\operatorname*{core}D\subseteq H\implies  \operatorname*{WEff}(F,D)\cap \operatorname*{dom}\varphi _{a-H,k}= \operatorname*{WEff}(F\cap\operatorname*{dom}\varphi _{a-H,k},D)$.
\end{itemize}
\end{lemma}
{\it Proof}
\begin{itemize}
\item[(a)] results from Lemma \ref{l-Eff-prop}(c).
\item[(b)] Consider an arbitrary $y^0\in F\cap\operatorname*{dom}\varphi _{a-H,k}$. $\Rightarrow \exists t\in\mathbb{R}:\;y^0\in a-H+tk$. Assume $y^0\notin\operatorname*{Eff}(F,D)$. $\Rightarrow \exists y\in F\cap (y^0-(D\setminus\{0\})$. $\Rightarrow y\in a-H+tk-D\subseteq a+tk-H\subseteq \operatorname*{dom}\varphi _{a-H,k}$.
$\Rightarrow y^0\notin\operatorname*{Eff}(F\cap\operatorname*{dom}\varphi _{a-H,k},D)$. 
\item[(c)] and (d) follow from (a) and (b) with $\operatorname*{core}D$ instead of $D$.
\qed
\end{itemize}

Let us first give some sufficient conditions for efficient and weakly efficient points by minimal solutions of functions $\varphi _{a-H,k}$. 

\begin{theorem}\label{hab-t421}
Define 
$$\Psi:=\operatorname*{argmin}_{y\in F\cap\operatorname*{dom}\varphi _{a-H,k}}\varphi _{a-H,k}(y).$$
Then:
\begin{itemize}
\item[(a)] $\operatorname*{Eff}(F,D)\cap \Psi\subseteq \operatorname*{Eff}(\Psi,D)$.
\item[(b)] $H+D\subseteq H\implies  \operatorname*{Eff}(F,D)\cap \Psi= \operatorname*{Eff}(\Psi,D)$.
\item[(c)] $H+D\subseteq H$ and $\Psi=\{ y^0\}$ imply $y^0\in\operatorname*{Eff}(F,D)$.
\item[(d)] If $H$ is $(-k)$-directionally closed and $H+(D\setminus\{0\})\subseteq \operatorname*{core}H$, then\\
$\Psi\subseteq \operatorname*{Eff}(F,D)$.
\item[(e)] $\Psi\subseteq \operatorname*{WEff}(F,D)$ holds if $H+D\subseteq H$ or if
$H$ is $(-k)$-directionally closed and $H+\operatorname*{core}D\subseteq \operatorname*{core}H$.
\end{itemize}
\end{theorem}
{\it Proof}
\begin{itemize}
\item[(a)] follows from Lemma \ref{l-Eff-prop}(c).
\item[(b)]$H+D\subseteq H\Rightarrow \varphi _{-H,k}$ is $D$-monotone by Theorem \ref{t251M}. $\Rightarrow \varphi _{a-H,k}$ is $D$-monotone because of Lemma \ref{A-shift}. $\Rightarrow \operatorname*{Eff}(F\cap\operatorname*{dom}\varphi _{a-H,k},D)\cap \Psi= \operatorname*{Eff}(\Psi,D)$ by Proposition
\ref{p-mon-scal}(b). This results in the assertion by Lemma \ref{hab-l421}(b).
\item[(c)] follows immediately from (b).
\item[(d)] If $\varphi _{a-H,k}$ is not finite-valued on $F\cap\operatorname*{dom}\varphi _{a-H,k}$, the assertion is fulfilled. Assume now that $\varphi _{a-H,k}$ is finite-valued on $F\cap\operatorname*{dom}\varphi _{a-H,k}$ and $H+(D\setminus\{0\})\subseteq \operatorname*{core}H$. Then $\varphi _{a-H,k}$ is strictly $D$-monotone on $F\cap\operatorname*{dom}\varphi _{a-H,k}$ by Theorem \ref{t251M}. $\Rightarrow \Psi\subseteq \operatorname*{Eff}(F\cap\operatorname*{dom}\varphi _{a-H,k},D)$ by Proposition \ref{p-mon-scal}(c). The assertion follows by Lemma \ref{hab-l421}(b).
\item[(e)] The second statement results from (d) with $D$ being replaced by $\operatorname*{core}D$.\\
If $\varphi _{a-H,k}$ is not finite-valued on $F\cap\operatorname*{dom}\varphi _{a-H,k}$, the first assertion is fulfilled as well. Assume now that $\varphi _{a-H,k}$ is finite-valued on $F\cap\operatorname*{dom}\varphi _{a-H,k}$ and that $H+D\subseteq H$ holds. By Theorem \ref{t251M}, $\varphi _{a-H,k}$ is strictly $(-\operatorname*{core}D)$-monotone on $F\cap\operatorname*{dom}\varphi _{a-H,k}$.  Proposition \ref{p-mon-wscal} implies $\Psi\subseteq \operatorname*{WEff}(F\cap\operatorname*{dom}\varphi _{a-H,k},D)$. By Lemma \ref{hab-l421},
$ \Psi\subseteq \operatorname*{WEff}(F,D)$.
\qed
\end{itemize}

\begin{remark}
Since $\varphi _{a-H,k}$ has been defined as an extended-real-valued functional,
\[ \operatorname*{min}_{y\in F\cap\operatorname*{dom}\varphi _{a-H,k}}\varphi _{a-H,k}(y)=\operatorname*{min}_{y\in F}\varphi _{a-H,k}(y).\]
We prefer to use the left formulation where we want to point out that $F$ is not necessarily contained in the effective domain of $\varphi _{a-H,k}$. In this case, $F\cap\operatorname*{dom}\varphi _{a-H,k}$ instead of $F$ is the feasible range of the optimization problem, which has immediate consequences for applications.
\end{remark}

\begin{example}
$Y=\mathbb{R}^2$, $H=D=\mathbb{R}_+^2+(1,1)^T$ and $k=(1,1)^T$ fulfill the assumptions $k\in 0^+H$, $H+D\subseteq H$ and $H+\operatorname*{core}D\subseteq \operatorname*{core}H$ though $D$ is not a convex cone.
\end{example}

\begin{corollary}\label{c-hab-t421}
Suppose that $D$ is a non-trivial convex cone in $Y$ and $k\in \operatorname*{core}D$. Define 
$$\Psi:=\operatorname*{argmin}_{y\in F}\varphi _{a-D,k}(y).$$
Then:
\begin{itemize}
\item[(a)] $\Psi\subseteq \operatorname*{WEff}(F,D)$.
\item[(b)] $\operatorname*{Eff}(F,D)\cap \Psi= \operatorname*{Eff}(\Psi,D)$.
\item[(c)] $\Psi=\{ y^0\} \implies y^0\in\operatorname*{Eff}(F,D)$.
\end{itemize}
\end{corollary}

We will now characterize the efficient point set and the weakly efficient point set by minimal solutions of functionals
$\varphi _{a-D,k}$. The following two theorems deliver necessary conditions for weakly efficient and for efficient elements.

\begin{theorem}\label{hab-p431a}
Assume $k\in 0^+D$ and that $D$ is $(-k)$-directionally closed.
\begin{eqnarray*}
\operatorname*{Eff}(F,D) & = & \{y^0\in F\mid \forall y\in (F\cap\operatorname*{dom}\varphi _{y^0-D,k})\setminus\{y^0\}\colon\varphi _{y^0-D,k}(y)>0\}\\
 & = & \{y^0\in F\mid \forall y\in (F\cap\operatorname*{dom}\varphi _{y^0-D,k})\setminus\{y^0\}\colon\varphi _{-D,k}(y-y^0)>0\}.
\end{eqnarray*} 
If $0\in D$, then $\varphi _{y^0-D,k}(y^0)=\varphi _{-D,k}(y^0-y^0)\leq 0$ for each $y^0\in Y$.\\
If $D+\mathbb{R}_>k\subseteq\operatorname*{core}D$ and $0\in D\setminus\operatorname*{core}D$, then\\
$\varphi _{y^0-D,k}(y^0)=\varphi _{-D,k}(y^0-y^0)=0$ for each $y^0\in Y$ and\\
$\operatorname*{Eff}(F,D)  =  \{y^0\in F\mid  \forall y\in (F\cap\operatorname*{dom}\varphi _{y^0-D,k})\setminus\{y^0\}\colon \varphi _{y^0-D,k}(y^0) < \varphi _{y^0-D,k}(y)\},$\\
$\operatorname*{Eff}(F,D)  =  \{y^0\in F\mid  \forall y\in (F\cap\operatorname*{dom}\varphi _{y^0-D,k})\setminus\{y^0\}\colon$\\
$\varphi _{-D,k}(y^0-y^0) < \varphi _{-D,k}(y-y^0)\}.$
\end{theorem}

\begin{proof}
For the relationship between $\varphi _{-D,k}$ and $\varphi _{y^0-D,k}$, see Lemma \ref{A-shift}.\\
 We now apply the statements from Theorem \ref{varphi-theo-allg}.\\
Consider some arbitrary $y^0\in F$. $y^0-D\subseteq \operatorname*{dom}\varphi _{y^0-D,k}$.\\
$y^0\in\operatorname*{Eff}(F,D)\Leftrightarrow F\cap (y^0-D)\subseteq\{y^0\}\Leftrightarrow\varphi _{y^0-D,k}(y)>0\mbox{ for all } y\in (F\cap\operatorname*{dom}\varphi _{y^0-D,k})\setminus\{y^0\}$.\\ 
$0\in D\Rightarrow y^0\in y^0-D.\Rightarrow \varphi _{y^0-D,k}(y^0)\leq 0$.\\
$D+\mathbb{R}_>k\subseteq\operatorname*{core}D$ and $0\in D\setminus\operatorname*{core}D$ imply
$\varphi _{y^0-D,k}(y^0)=0$ by Theorem \ref{varphi-theo-allg} and thus the assertion.
\end{proof}

Because of Theorem \ref{t-K4-cons},  the functionals $\varphi _{y^0-D,k}$ in Theorem \ref{hab-p431a} are convex if and only if $D$ is a convex set, and the functional $\varphi _{-D,k}$ is sublinear if and only if $D$ is a convex cone. Note that $\varphi _{-D,k}$ and each $\varphi _{y^0-D,k}$ are finite-valued, if $k\in\operatorname*{core}0^+D$.

In Theorem \ref{hab-p431a}, efficient elements $y^0$ are described as unique minimizers of $\varphi _{y^0-D,k}$. Without the uniqueness, we get weakly efficient points.

\begin{theorem}\label{hab-p431b}
Assume $D+\mathbb{R}_>k\subseteq\operatorname*{core}D$. 
\begin{eqnarray*}
\operatorname*{WEff}(F,D) & = & \{y^0\in F\mid \forall y\in (F\cap\operatorname*{dom}\varphi _{y^0-D,k})\setminus\{y^0\}\colon \varphi _{y^0-D,k}(y)\geq 0\}\\
 & = & \{y^0\in F\mid \forall y\in (F\cap\operatorname*{dom}\varphi _{y^0-D,k})\setminus\{y^0\}\colon\varphi _{-D,k}(y-y^0)\geq 0\}.
\end{eqnarray*} 
Assume now, additionally, that $D$ is $(-k)$-directionally closed.\\
If $0\in D$, then $\varphi _{y^0-D,k}(y^0)=\varphi _{-D,k}(y^0-y^0)\leq 0$ for each $y^0\in Y$.\\
If $0\in D\setminus\operatorname*{core}D$, then $\varphi _{y^0-D,k}(y^0)=\varphi _{-D,k}(y^0-y^0)=0$ for each $y^0\in Y$ and
\begin{eqnarray*}
\operatorname*{WEff}(F,D) & = & \{y^0\in F\mid \varphi _{y^0-D,k}(y^0)=\operatorname*{min}_{y\in F\cap\operatorname*{dom}\varphi _{y^0-D,k}} \varphi _{y^0-D,k}(y)\}\\
 & = & \{y^0\in F\mid \varphi _{-D,k}(y^0-y^0)=\operatorname*{min}_{y\in F\cap\operatorname*{dom}\varphi _{y^0-D,k}} \varphi _{-D,k}(y-y^0)\}.
\end{eqnarray*} 
\end{theorem}

\begin{proof}
For the relation between $\varphi _{-D,k}$ and $\varphi _{y^0-D,k}$, see Lemma \ref{A-shift}.\\
Consider some arbitrary $y^0\in F$. $y^0-D\subseteq \operatorname*{dom}\varphi _{y^0-D,k}$.\\
$y^0\in\operatorname*{WEff}(F,D)\Leftrightarrow F\cap (y^0-\operatorname*{core}D)\subseteq\{y^0\}\Leftrightarrow\varphi _{y^0-D,k}(y)\geq 0\mbox{ for all } y\in (F\cap\operatorname*{dom}\varphi _{y^0-D,k})\setminus\{y^0\}$.\\ 
This and the statements about $\varphi _{y^0-D,k}(y^0)$ for a $(-k)$-directionally closed set $D$ follow from Theorem \ref{varphi-theo-allg}. This results in the assertions.
\end{proof}

Up to now, we have used functionals $\varphi _{y^0-D,k}$ for scalarizing the weakly efficient point set and the efficient point set, where $y^0$ was the (weakly) efficient element. We now turn to scalarization by functions $\varphi _{a-D,k}$, where $a$ is a fixed vector. In this case, $a$ can be a lower or an upper bound of $F$ and $D$ has to be a convex cone. In applications, an upper bound can easily be added to the vector optimization problem without any influence on the set of solutions. Note that scalarizations that are based on norms require a lower bound  of $F$. 

\begin{theorem}\label{hab-p432}
Suppose that $D$ is a non-trivial algebraically closed convex cone with $\operatorname*{core}D\not=\emptyset$.\\
If $F\subseteq a-\operatorname*{core}D$ or $F\subseteq a+\operatorname*{core}D$ for some $a\in Y$, then
$$ \operatorname*{WEff}(F,D)=\{y^0\in F\mid \exists k\in\operatorname*{core}D:\; \varphi_{a-D,k}(y^0)=\min_{y\in F}\varphi_{a-D,k}(y)\} \quad\mbox{ and }$$
$$ \operatorname*{Eff}(F,D)=\{y^0\in F\mid \exists k\in\operatorname*{core}D\;\forall y\in F\setminus\{y^0\}:\; \varphi_{a-D,k}(y^0)<\varphi_{a-D,k}(y)\}.$$
For $y^0\in \operatorname*{WEff}(F,D)$, in the case $F\subseteq a-\operatorname*{core}D$ one can choose $k=a-y^0$, which results in 
$\varphi_{a-D,k}(y^0)=-1$, whereas in the case $F\subseteq a+\operatorname*{core}D$ one can choose $k=y^0-a$, which results in 
$\varphi_{a-D,k}(y^0)=1$.
\end{theorem}

\begin{proof}
Because of Corollary \ref{c-hab-t421}, we have only to show the inclusions $\subseteq$ of the equations.\\
Assume $F\subseteq a-\operatorname*{core}D$ and consider some $y^0\in \operatorname*{WEff}(F,D)$.\\
$\Rightarrow y^0\in F\subseteq a-\operatorname*{core}D$. $\Rightarrow k:=a-y^0\in \operatorname*{core}D$.\\
$\varphi_{a-D,k}(y)=-1\Leftrightarrow y\in a-(D\setminus\operatorname*{core}D)-k=y^0-(D\setminus\operatorname*{core}D).$
Thus $0\in D\setminus\operatorname*{core}D$ implies $\varphi_{a-D,k}(y^0)=-1$.\\
$\varphi_{a-D,k}(y)<-1\Leftrightarrow y\in a-\operatorname*{core}D-k=y^0-\operatorname*{core}D.$\\
$F\cap (y^0-\operatorname*{core}D)\subseteq\{y^0\}$ implies $\varphi_{a-D,k}(y)\not< -1\mbox{ for all } y\in F\setminus\{y^0\}$, thus
$\varphi_{a-D,k}(y^0)=\min_{y\in F}\varphi_{a-D,k}(y)$.\\
For $y^0\in \operatorname*{Eff}(F,D)$, $F\cap (y^0-D)=\{y^0\}$ and hence $\varphi_{a-D,k}(y)\not\leq -1\mbox{ for all } y\in F\setminus\{y^0\}$.\\
The case $F\subseteq a+\operatorname*{core}D$ can be handled in an analogous way. 
\end{proof}
\smallskip

\section{Scalarization in Vector Optimization by Norms}\label{s-scal-norm}

Because of Proposition \ref{p-ordint-varphi}, Section \ref{s-scal-uni} delivers scalarization results for efficient and weakly efficient elements by norms. 

Let us first give some sufficient conditions for efficient and weakly efficient points by minimal solutions of norms. 
Theorem \ref{hab-t421} implies with Proposition \ref{p-ordint-varphi}:

\begin{theorem}\label{n-hab-t421}
Suppose $C$ is a non-trivial algebraically closed convex pointed cone in $Y$ with
$k\in \operatorname*{core}C$ and $F\subseteq a+C$.
Define
$$\Psi := \operatorname*{argmin}_{y\in F} \| y-a\| _{C,k}.$$ 
Then:
\begin{itemize}
\item[(a)] $\operatorname*{Eff}(F,D)\cap \Psi\subseteq \operatorname*{Eff}(\Psi,D)$.
\item[(b)] $C+D\subseteq C\implies  \operatorname*{Eff}(F,D)\cap \Psi= \operatorname*{Eff}(\Psi,D)$.
\item[(c)] $C+D\subseteq C$ and $\Psi=\{ y^0\}$ imply $y^0\in\operatorname*{Eff}(F,D)$.
\item[(d)] $C+(D\setminus\{0\})\subseteq \operatorname*{core}C \implies\Psi\subseteq \operatorname*{Eff}(F,D)$.
\item[(e)] $C+\operatorname*{core}D\subseteq \operatorname*{core}C \implies  \Psi\subseteq \operatorname*{WEff}(F,D)$.
\end{itemize}
\end{theorem}

\begin{corollary}\label{n-c-hab-t421}
Suppose $D$ is a non-trivial algebraically closed convex pointed cone with
$k\in \operatorname*{core}D$ and $F\subseteq a+D$.
Define
$$\Psi := \operatorname*{argmin}_{y\in F} \| y-a\| _{D,k}.$$
Then:
\begin{itemize}
\item[(a)] $\Psi\subseteq \operatorname*{WEff}(F,D)$.
\item[(b)] $\operatorname*{Eff}(F,D)\cap \Psi= \operatorname*{Eff}(\Psi,D)$.
\item[(c)] $\Psi=\{ y^0\}$ implies $y^0\in\operatorname*{Eff}(F,D)$.
\end{itemize}
\end{corollary}

We will now characterize the efficient point set and the weakly efficient point set by minimal solutions of norms. 
We get from Theorem \ref{hab-p432} by Proposition \ref{p-ordint-varphi}:

\begin{theorem}\label{n-hab-p432}
Suppose that $D$ is a non-trivial algebraically closed convex pointed cone and
$F\subseteq a+\operatorname*{core}D$. Then
$$ \operatorname*{WEff}(F,D)=\{y^0\in F\mid \exists k\in\operatorname*{core}D:\; \| y^0-a\| _{D,k}=\operatorname*{min}_{y\in F}\| y-a\| _{D,k}\} \quad\mbox{ and }$$
$$ \operatorname*{Eff}(F,D)=\{y^0\in F\mid \exists k\in\operatorname*{core}D\;\forall y\in F\setminus\{y^0\}:\; \| y^0-a\| _{D,k}<\| y-a\| _{D,k}\}.$$
For $y^0\in \operatorname*{WEff}(F,D)$, one can choose $k=y^0-a$, which results in 
$\| y^0-a\| _{D,k}=1$.
\end{theorem}
\smallskip

\def\cfac#1{\ifmmode\setbox7\hbox{$\accent"5E#1$}\else
  \setbox7\hbox{\accent"5E#1}\penalty 10000\relax\fi\raise 1\ht7
  \hbox{\lower1.15ex\hbox to 1\wd7{\hss\accent"13\hss}}\penalty 10000
  \hskip-1\wd7\penalty 10000\box7}
  \def\cfac#1{\ifmmode\setbox7\hbox{$\accent"5E#1$}\else
  \setbox7\hbox{\accent"5E#1}\penalty 10000\relax\fi\raise 1\ht7
  \hbox{\lower1.15ex\hbox to 1\wd7{\hss\accent"13\hss}}\penalty 10000
  \hskip-1\wd7\penalty 10000\box7}
  \def\cfac#1{\ifmmode\setbox7\hbox{$\accent"5E#1$}\else
  \setbox7\hbox{\accent"5E#1}\penalty 10000\relax\fi\raise 1\ht7
  \hbox{\lower1.15ex\hbox to 1\wd7{\hss\accent"13\hss}}\penalty 10000
  \hskip-1\wd7\penalty 10000\box7}
  \def\cfac#1{\ifmmode\setbox7\hbox{$\accent"5E#1$}\else
  \setbox7\hbox{\accent"5E#1}\penalty 10000\relax\fi\raise 1\ht7
  \hbox{\lower1.15ex\hbox to 1\wd7{\hss\accent"13\hss}}\penalty 10000
  \hskip-1\wd7\penalty 10000\box7}
  \def\cfac#1{\ifmmode\setbox7\hbox{$\accent"5E#1$}\else
  \setbox7\hbox{\accent"5E#1}\penalty 10000\relax\fi\raise 1\ht7
  \hbox{\lower1.15ex\hbox to 1\wd7{\hss\accent"13\hss}}\penalty 10000
  \hskip-1\wd7\penalty 10000\box7}
  \def\cfac#1{\ifmmode\setbox7\hbox{$\accent"5E#1$}\else
  \setbox7\hbox{\accent"5E#1}\penalty 10000\relax\fi\raise 1\ht7
  \hbox{\lower1.15ex\hbox to 1\wd7{\hss\accent"13\hss}}\penalty 10000
  \hskip-1\wd7\penalty 10000\box7}
  \def\cfac#1{\ifmmode\setbox7\hbox{$\accent"5E#1$}\else
  \setbox7\hbox{\accent"5E#1}\penalty 10000\relax\fi\raise 1\ht7
  \hbox{\lower1.15ex\hbox to 1\wd7{\hss\accent"13\hss}}\penalty 10000
  \hskip-1\wd7\penalty 10000\box7}
  \def\cfac#1{\ifmmode\setbox7\hbox{$\accent"5E#1$}\else
  \setbox7\hbox{\accent"5E#1}\penalty 10000\relax\fi\raise 1\ht7
  \hbox{\lower1.15ex\hbox to 1\wd7{\hss\accent"13\hss}}\penalty 10000
  \hskip-1\wd7\penalty 10000\box7}
  \def\cfac#1{\ifmmode\setbox7\hbox{$\accent"5E#1$}\else
  \setbox7\hbox{\accent"5E#1}\penalty 10000\relax\fi\raise 1\ht7
  \hbox{\lower1.15ex\hbox to 1\wd7{\hss\accent"13\hss}}\penalty 10000
  \hskip-1\wd7\penalty 10000\box7} \def\Dbar{\leavevmode\lower.6ex\hbox to
  0pt{\hskip-.23ex \accent"16\hss}D}
  \def\cfac#1{\ifmmode\setbox7\hbox{$\accent"5E#1$}\else
  \setbox7\hbox{\accent"5E#1}\penalty 10000\relax\fi\raise 1\ht7
  \hbox{\lower1.15ex\hbox to 1\wd7{\hss\accent"13\hss}}\penalty 10000
  \hskip-1\wd7\penalty 10000\box7} \def\cprime{$'$}
  \def\Dbar{\leavevmode\lower.6ex\hbox to 0pt{\hskip-.23ex \accent"16\hss}D}
  \def\cfac#1{\ifmmode\setbox7\hbox{$\accent"5E#1$}\else
  \setbox7\hbox{\accent"5E#1}\penalty 10000\relax\fi\raise 1\ht7
  \hbox{\lower1.15ex\hbox to 1\wd7{\hss\accent"13\hss}}\penalty 10000
  \hskip-1\wd7\penalty 10000\box7} \def\cprime{$'$}
  \def\Dbar{\leavevmode\lower.6ex\hbox to 0pt{\hskip-.23ex \accent"16\hss}D}
  \def\cfac#1{\ifmmode\setbox7\hbox{$\accent"5E#1$}\else
  \setbox7\hbox{\accent"5E#1}\penalty 10000\relax\fi\raise 1\ht7
  \hbox{\lower1.15ex\hbox to 1\wd7{\hss\accent"13\hss}}\penalty 10000
  \hskip-1\wd7\penalty 10000\box7}
  \def\udot#1{\ifmmode\oalign{$#1$\crcr\hidewidth.\hidewidth
  }\else\oalign{#1\crcr\hidewidth.\hidewidth}\fi}
  \def\cfac#1{\ifmmode\setbox7\hbox{$\accent"5E#1$}\else
  \setbox7\hbox{\accent"5E#1}\penalty 10000\relax\fi\raise 1\ht7
  \hbox{\lower1.15ex\hbox to 1\wd7{\hss\accent"13\hss}}\penalty 10000
  \hskip-1\wd7\penalty 10000\box7} \def\Dbar{\leavevmode\lower.6ex\hbox to
  0pt{\hskip-.23ex \accent"16\hss}D}
  \def\cfac#1{\ifmmode\setbox7\hbox{$\accent"5E#1$}\else
  \setbox7\hbox{\accent"5E#1}\penalty 10000\relax\fi\raise 1\ht7
  \hbox{\lower1.15ex\hbox to 1\wd7{\hss\accent"13\hss}}\penalty 10000
  \hskip-1\wd7\penalty 10000\box7} \def\Dbar{\leavevmode\lower.6ex\hbox to
  0pt{\hskip-.23ex \accent"16\hss}D}
  \def\cfac#1{\ifmmode\setbox7\hbox{$\accent"5E#1$}\else
  \setbox7\hbox{\accent"5E#1}\penalty 10000\relax\fi\raise 1\ht7
  \hbox{\lower1.15ex\hbox to 1\wd7{\hss\accent"13\hss}}\penalty 10000
  \hskip-1\wd7\penalty 10000\box7} \def\Dbar{\leavevmode\lower.6ex\hbox to
  0pt{\hskip-.23ex \accent"16\hss}D}
  \def\cfac#1{\ifmmode\setbox7\hbox{$\accent"5E#1$}\else
  \setbox7\hbox{\accent"5E#1}\penalty 10000\relax\fi\raise 1\ht7
  \hbox{\lower1.15ex\hbox to 1\wd7{\hss\accent"13\hss}}\penalty 10000
  \hskip-1\wd7\penalty 10000\box7}

\end{document}